\documentclass[11pt]{article}
\usepackage{amssymb, amsmath}

\usepackage[margin=1.2in]{geometry}

\usepackage{mathptmx}

\newtheorem{theorem}{Theorem}[section]
\newtheorem{lemma}[theorem]{Lemma}

\newtheorem{corollary}[theorem]{Corollary}
\newtheorem{conjecture}[theorem]{Conjecture}

\newenvironment{proof}[1][Proof]{\begin{trivlist}
\item[\hskip \labelsep {\bfseries #1}]}{\end{trivlist}}
\newenvironment{definition}[1][Definition]{\begin{trivlist}
\item[\hskip \labelsep {\bfseries #1}]}{\end{trivlist}}

\newenvironment{remark}[1][Remark]{\begin{trivlist}
\item[\hskip \labelsep {\bfseries #1}]}{\end{trivlist}}

\newcommand{\qed}{\nobreak \ifvmode \relax \else
      \ifdim\lastskip<1.5em \hskip-\lastskip
      \hskip1.5em plus0em minus0.5em \fi \nobreak
      \vrule height0.75em width0.5em depth0.25em\fi}

\renewcommand{\Re}{\operatorname{Re}}

\begin{document}
\title{On Barnes Beta Distributions and Applications to the Maximum Distribution of the 2D Gaussian Free Field}

\author{Dmitry Ostrovsky}

\date{July 21, 2016}

\maketitle
\noindent

\begin{abstract}
\noindent
A new family of Barnes beta distributions on $(0, \infty)$ is introduced
and its infinite divisibility, moment determinacy, scaling, and factorization properties are established. 
The Morris integral probability distribution is constructed from Barnes beta distributions 
of types $(1,0)$ and $(2,2),$ and its moment determinacy and involution invariance properties
are established. For application, the maximum distributions of the 2D gaussian free field
on the unit interval and circle with a non-random logarithmic potential are conjecturally related to the critical Selberg and Morris integral probability 
distributions, respectively, and expressed in terms of sums of Barnes beta distributions of types $(1,0)$ and $(2,2).$
\end{abstract}

\section{Introduction}
\noindent This paper advances the theory of Barnes beta probability distributions that we introduced in \cite{Me13} and \cite{Me14}. 
Barnes beta distributions lie at the crossroads
of several areas of probability theory and statistical physics: Dufresne distributions \cite{Duf10}, infinite divisibility in the context of special 
functions of analytic number theory \cite{BiaPitYor}, \cite{Jacod}, \cite{NikYor}, L\'evy processes \cite{Kuz} -- \cite{LetSim},
and, conjecturally, gaussian multiplicative chaos \cite{MeIMRN}, 
the distribution of the maximum of the characteristic polynomial of GUE random matrices \cite{FyodSimm}, 
and mod-gaussian limit theorems for mesoscopic statistics \cite{Menon}.

The defining property of Barnes beta distributions $\beta_{M, N}(a, b),$ $M\leq N,$ $a=(a_1,\cdots, a_M),$ 
$b=(b_0, b_1,\cdots, b_N),$ 
is that their Mellin transform is 
given in the form of an intertwining product of ratios of multiple gamma functions of Barnes \cite{mBarnes}. For example,
in the special case of type (2,2), $a=(a_1, a_2),$ $b=(b_0, b_1, b_2),$ the Mellin transform is
\begin{align}
{\bf E}\bigl[\beta_{2,2}(a, b)^q\bigr] = &
\frac{\Gamma_2(q+b_0\,|\,a)}{\Gamma_2(b_0\,|\,a)}
\frac{\Gamma_2(b_0+b_1\,|\,a)}{\Gamma_2(q+b_0+b_1\,|\,a)}
\frac{\Gamma_2(b_0+b_2\,|\,a)}{\Gamma_2(q+b_0+b_2\,|\,a)}\times \nonumber \\ & \times  
\frac{\Gamma_2(q+b_0+b_1+b_2\,|\,a)}{\Gamma_2(b_0+b_1+b_2\,|\,a)}. \label{basic}
\end{align}
%%One observes that $q$ comes only with the '+' sign in this formula and it is known that $\beta_{2 , 2}$ takes value in $(0, 1).$ 
%%The first property is true of the Mellin transform of general $\beta_{M, N},$ whereas 
It is known that $\beta_{M, N}(a, b)$ takes values in (0, 1] if $M<N$ and in $(0,\,1)$ if $M=N.$

The contribution of this paper is to review the general theory of Barnes
beta distributions for $M\leq N,$ extend it to the case of $M=N+1,$ and give novel applications of the new and existing theory. 
Our theoretical contribution is %%the extension of the original construction
%%to a new family of distributions that are supported on $(0, \infty).$ 
the proof that $\beta_{M, M-1}(a, b)$ exists for $M\in\mathbb{N},$ takes values in $(0, \infty),$
has infinitely divisible logarithm, and satisfies remarkable infinite factorizations.
In particular, the new construction specializes to certain intertwining products of multiple sine functions.
This is motivated by the work of \cite{Kuz} -- \cite{LetSim}, who computed the Mellin transform of certain 
functionals of the stable L\'evy process in the form of products of ratios of double gamma functions
%%similar to \eqref{basic} except having '+' and '-' signs in the numerator of denominator. 
that are different from \eqref{basic}. %%We show that 
Their distributions coincide with what we call $\beta_{2, 1}(a, b),$ 
$a=(a_1, a_2),$ $b=(b_0, b_1),$
\begin{equation}\label{basicL}
{\bf E}\bigl[\beta_{2,1}(a, b)^q\bigr] =
\frac{\Gamma_2(q+b_0\,|\,a)}{\Gamma_2(b_0\,|\,a)}
\frac{\Gamma_2(b_0+b_1\,|\,a)}{\Gamma_2(q+b_0+b_1\,|\,a)}.
\end{equation}
%%and then extend this construction to arbitrary multiple gamma functions. 

The contribution of this paper to statistical physics is to formulate precise conjectures about the distribution of the maximum of 
the 2D gaussian free field (GFF) on the unit interval and circle with a non-random logarithmic potential in terms of Barnes beta distributions. %%of type (2,2). 
The GFF is a fascinating mathematical object that is of fundamental interest in statistical mechanics of disordered energy landscapes \cite{CRS}, \cite{FyoBou}, \cite{YK},
\cite{FLD}, \cite{FLDR}, \cite{FLDR2}, quantum disordered systems exhibiting multifractality \cite{CLD}, \cite{CMW}, \cite{YO}, and quantum gravity \cite{DS}, \cite{RV}. 
It also appears naturally in 
the limit of a wide class of statistics. For example, \cite{Hughesetal} proved convergence of
the distribution of the log-characteristic polynomial of large unitary matrices to the GFF on the circle
and \cite{FKS} proved convergence of an appropriately rescaled log-characteristic polynomial of large Hermitian matrices
to the GFF on the interval. We recently showed in \cite{Menon} that the smoothed indicator function of the mesoscopic statistics of Riemann zeroes
of Bourgade-Kuan \cite{BK} and Rodgers \cite {Rodg} converges to the GFF on the interval. We also indicated there that the same result is true of
any linear statistic that converges to a gaussian process having $\mathcal{H}^{1/2}(\mathbb{R})$ limiting covariance
such as the local CUE statistic of Soshnikov \cite{Sosh}, for example. %%more generally, the large class of linear mesoscopic
%%statistics of random matrices ...(Gaulitier). 

The problem of calculating the maximum distribution of the GFF 
was first considered in \cite{FyoBou} for the GFF on the circle and in \cite{FLDR} on the interval, and then extended %% a series of publications
in \cite{YO}, \cite{FLDR2}, and most recently in  \cite{CRS} and \cite{FLD}. In all cases they conjectured 
the Laplace transform of the maximum distribution including the leading asymptotic (non-random) drift and the fluctuating part of
the distribution. While there has recently been made good progress in verifying their results for the drift, see \cite{Ding},
their conjectures for the fluctuating part are still beyond the reach of existing methods, as the law of the so-called derivative martingale is unknown, 
see \cite{Madmax} and \cite{SubZei} for details. The difficulty of the problem is
that their solutions are based on the still unproven freezing scenario for the gaussian multiplicative chaos measure, see \cite{Mad}
for recent progress, and on the conjectured form of the Mellin transform of the total (random) mass of that measure,
see \cite{Webb} for partial results in the circle case. Our contribution is to show that the result of \cite{FLD} and \cite{FLDR} for the fluctuating part 
on the interval does correspond
to a valid probability distribution, namely, the critical Selberg integral probability distribution, which we developed in 
\cite{MeIMRN} and \cite{Me14}, and can be naturally expressed in terms of Barnes beta 
distributions of types $(1,0)$ and $(2,2),$ thereby conjecturing the law of the derivative martingale. 
%%see \cite{FyoBou}, \cite{FLDR}, \cite{Me4} \cite{MeIMRN}. 
%%Our contribution is to show that their result for the interval can be naturally expressed in terms of Barnes beta
%%distributions of types $(1,0)$ and $(2,2)$ and, in particular, the Selberg integral probability distribution that we developed in 
%%\cite{MeIMRN} and \cite{Me14}. 
We also clarify the origin of the self-duality of the Mellin transform that plays a key role   
in their freezing framework and give a heuristic derivation of the conjecture by combining
their calculations for the fluctuating part with our conjecture about the total mass of the
multiplicative chaos measure on the interval. 
%%provide a self-contained review of their calculations for the fluctuating part with an 
%%emphasis on the connection with the theory of gaussian multiplicative chaos. 
In the case of the circle, we construct a new probability distribution having the property that its $n$th moment equals 
the value of the Morris integral of dimension $n,$ express it in terms of Barnes beta distributions of types $(1,0)$ and $(2,2),$ 
prove self-duality of its Mellin transform, and further extend the original work of \cite{FyoBou} by formulating 
a novel conjecture about  (the fluctuating part of) the distribution of the maximum of the GFF on the circle
with a non-random logarithmic potential
in terms of the critical Morris integral distribution, thereby also conjecturing the law of the 
derivative martingale in this case.
%%conjecturally relate it to This advances

The plan of the paper is as follows. In Section 2 we give a review of the general theory of Barnes beta distributions of type $(M, N)$
and of the Selberg integral distribution as a special case of the theory corresponding to type (2,2). In Section 3 we construct a new family 
of Barnes beta distributions on $(0, \infty).$ In Section 4 we construct the Morris integral distribution and study its properties.
%%In Section 5 we briefly review the gaussian free field on the interval and circle, formulate our conjectures about its maximum distribution,
%%heuristic derivation of the conjectures. 
In Section 5 we treat the maximum distribution of the gaussian free field. In Section 6 we give the proofs. Section 7 concludes.

\section{A Review of Multiple Gamma Functions and Barnes Beta and Selberg Integral Distributions} %%, and Selberg Integral Distribution}
\noindent
The multiple gamma function of Barnes \cite{mBarnes} is defined classically by
\begin{equation}\label{barnes}
\Gamma^{-1}_M(w\,|\,a) = e^{P(w\,|\,a)}\,w\prod\limits_{n_1,\cdots , n_M=0}^\infty
{}' \Bigl(1+\frac{w}{\Omega}\Bigr)\exp\Bigl(\sum_{k=1}^M \frac{(-1)^k}{k}\frac{w^k}{\Omega^k}\Bigr), 
\end{equation}
where $\Re(w)>0,$ $P(w\,|\,a)$ is a polynomial in $w$ of degree $M$ that depends on one's choice of
normalization, the parameters $a_j>0,$ $j=1\cdots M,$ 
\begin{equation}\label{Omega}
\Omega\triangleq \sum_{i=1}^M n_i \, a_i,
\end{equation}
and the prime indicates that the product is over all indices except $n_1=\cdots =n_M=0.$
Our choice of normalization is that of Ruijsenaars \cite{Ruij}, which is explained next. %%, which is equivalent to defining the logarithm the multiple gamma function
Let
\begin{equation}\label{fdef}
f_M(t|a) = t^M \prod\limits_{j=1}^M (1-e^{-a_j t})^{-1},
\end{equation}
for some integer $M\geq 0.$
%%\footnote{Many of our results can be formulated for a general
%%function of the Ruijsenaars class, \emph{i.e.} analytic for
%%$\Re(t)>0$ and at $t=0$ and of at worst polynomial growth as
%%$t\rightarrow \infty,$ see \cite{Ruij}, Section 2. }
Slightly
modifying the definition in \cite{Ruij}, we define multiple
Bernoulli polynomials by
\begin{equation}\label{Bdefa}
B_{M, m}(x|a) \triangleq \frac{d^m}{dt^m}\Big|_{t=0} \bigl[f_M(t|a)
e^{-xt}\bigr].
%%\sum\limits_{k=0}^m \binom{m}{k}\,f^{(k)}(0)
%%\,(-x)^{m-k}, \,\, m\in\mathbb{N}. %% , \;\text{or equivalently},\label{Bdefa} \\
%%f(t)\,e^{-xt} & = \sum\limits_{n=0}^\infty B^{(f)}_n(x)
%%\frac{t^n}{n!}, \label{Bdefb}
\end{equation}
%%for all sufficiently small $t.$
%%and in the case of \eqref{fdef} denote them by \(B_{M,m}(x\,|\,a).\)
The key result of \cite{Ruij} about the log-multiple Gamma function that
we need is summarized in the following theorem. 
\begin{theorem}[Ruijsenaars]\label{R}
$\log\Gamma_M(w|a)$ satisfies the Malmst\'en-type formula for $\Re(w)>0,$
\begin{equation}\label{keyRuij}
\log\Gamma_M(w|a) = \int\limits_0^\infty \frac{dt}{t^{M+1}} \Bigl(
e^{-wt}\,f_M(t|a) - \sum\limits_{k=0}^{M-1} \frac{t^k}{k!}\,B_{M,k}(w|a)
- \frac{t^M\,e^{-t}}{M!}\, B_{M, M}(w|a)\Bigr).
\end{equation}
$\log\Gamma_M(w\,|\,a)$ satisfies the asymptotic expansion,
\begin{align}
\log\Gamma_M(w\,|\,a) = & -\frac{1}{M!} B_{M, M}(w\,|\,a)\,\log(w) + \sum\limits_{k=0}^M
\frac{B_{M,k}(0\,|\,a) (-w)^{M-k}}{k!(M-k)!}\sum\limits_{l=1}^{M-k}
\frac{1}{l} + \nonumber \\ & +  R_M(w\,|\,a), \label{asym}\\
R_M(w\,|\,a) = &  O(w^{-1}), \,|w|\rightarrow\infty, \, |\arg(w)|<\pi.
\label{asymremainder}
\end{align}
\end{theorem}
The formula in \eqref{keyRuij} can be thought of as defining the Barnes multiple gamma function by
%%Now, following \cite{Ruij}, define 
\begin{equation}\label{mgamma}
\Gamma_M(w\,|\,a) \triangleq \exp\bigl(\log\Gamma_M(w\,|\,a)\bigr),
\end{equation}
in which case one proves that the formula in \eqref{barnes} holds for a particular choice of $P(w\,|\,a).$
One can show that
$\Gamma_M(w\,|\,a)$ satisfies the fundamental functional equation
\begin{equation}\label{feq}
\Gamma_{M}(w\,|\,a) =
\Gamma_{M-1}(w\,|\,\hat{a}_i)\,\Gamma_M\bigl(w+a_i\,|\,a\bigr),\,i=1\cdots
M, \,\,M\in\mathbb{N},
\end{equation}
$\hat{a}_i = (a_1,\cdots, a_{i-1},\,a_{i+1},\cdots, a_{M}).$ 
For example, 
\begin{align}
\Gamma_1(w\,|\,a) & = \frac{a^{w/a-1/2}}{\sqrt{2\pi}} \,\Gamma(w/a), \label{gamma1}\\
\Gamma_0(w) &= 1/w. %%which is also due to \cite{mBarnes}.
\end{align}

The multiple gamma function has three additional properties that are important for our purposes.
%%\begin{theorem}[Scaling invariance]\label{scaling}
Let \(\Re(w)>0,\) \(\kappa>0\) and \((\kappa\,a)_i\triangleq\kappa\,a_i,\;i=1\cdots M.\)
Then, it has the scaling property,
\begin{equation}\label{scale}
\Gamma_M(\kappa w\,|\,\kappa a) = \kappa^{-B_{M,M}(w\,|\,a)/M!}\,\Gamma_M(w\,|\,a).
\end{equation}
%%\end{theorem}
Let $\Re(w)>0$ and $k=1,2,3,\cdots.$ It has the multiplication property,
\begin{equation}\label{multiply}
\Gamma_M(kw\,|\,a) = k^{-B_{M,
M}(kw\,|\,a)/M!}\,\prod\limits_{p_1,\cdots,p_M=0}^{k-1}\Gamma_M\Bigl(w+\frac{\sum_{j=1}^M
p_j a_j}{k}\,\Big|\,a\Bigr).
\end{equation}
Given $x>0$ and \(a=(a_1\cdots a_{M-1}),\)
there exist functions $\phi_{M}\bigl(w,x\,|\,a, a_{M}\bigr)$ and \\ \(\Psi_{M}(w,y\,|\,a)\) such that
\begin{align}
\Gamma_{M}\bigl(w\,|\,a,a_{M}\bigr) = &  e^{\phi_{M}(w,x\,|\,a, a_{M})}\, \Gamma_{M-1}(w\,|\,a)\prod\limits_{k=1}^\infty \frac{\Gamma_{M-1}(w+ka_{M}\,|\,a)}{\Gamma_{M-1}(x+ka_{M}\,|\,a)}\times \nonumber \\ & \times \exp\Bigl(\Psi_{M}(x,ka_{M}\,|\,a)-\Psi_{M}(w,ka_{M}\,|\,a)\Bigr). \label{generalfactorization}
%%\star \nonumber \\
%%&\star\exp{\bigl(\phi_{M}(w,x\,|\,a, a_{M})\bigr)}\, \Gamma_{M-1}(w\,|\,a). 
\end{align}
\(\Psi_{M}(w,y\,|\,a)\) and \(\phi_{M}(w,x\,|\,a, a_{M})\) are polynomials in $w$ of degree $M.$ This is known as Shintani
factorization, see \cite{Shintani} for the original result for $M=2.$ The interested reader can find explicit formulas for
\(\Psi_{M}(w,y\,|\,a)\) and \(\phi_{M}(w,x\,|\,a, a_{M})\) and derivations of all the three properties in \cite{Me14}.

We now proceed to review the Barnes beta construction following \cite{Me13} and \cite{Me14}. 
Define the action of the combinatorial operator $\mathcal{S}_N$ on a function \(h(x)\) by
\begin{definition}\label{Soperator}
\begin{equation}\label{S}
(\mathcal{S}_Nh)(q\,|\,b) \triangleq \sum\limits_{p=0}^N (-1)^p
\sum\limits_{k_1<\cdots<k_p=1}^N
h\bigl(q+b_0+b_{k_1}+\cdots+b_{k_p}\bigr).
\end{equation}
\end{definition}
\noindent In other words, in \eqref{S} the action of $\mathcal{S}_N$
is defined as an alternating sum over all combinations of $p$
elements for every $p=0\cdots N.$ 
\begin{definition}\label{bdef}
Given $q\in\mathbb{C}-(-\infty, -b_0],$ \(a=(a_1\cdots a_{M})\), \(b=(b_0,b_1\cdots b_{N}),\) let\footnote{
We will abbreviate $\bigl(\mathcal{S}_N \log\Gamma_M\bigr)(q\,|a,\,b)$ to mean the action of $\mathcal{S}_N$ on
$\log\Gamma_M(x|a),$ \emph{i.e.} $\bigl(\mathcal{S}_N \log\Gamma_M(x|a)\bigr)(q\,|\,b).$}
\begin{equation}\label{eta}
\eta_{M,N}(q\,|a,\,b) \triangleq \exp\Bigl(\bigl(\mathcal{S}_N
\log\Gamma_M\bigr)(q\,|a,\,b) - \bigl(\mathcal{S}_N \log\Gamma_M\bigr)(0\,|a,\,b)\Bigr).
\end{equation}
\end{definition}
The function $\eta_{M,N}(q\,|a,\,b)$ is holomorphic over
$q\in\mathbb{C}-(-\infty, -b_0]$ and equals a product of ratios of
multiple gamma functions by construction. Specifically,
\begin{align}
\eta_{M,N}(q\,|a, b) 
= & \frac{\Gamma_M(q+b_0|a)}{\Gamma_M(b_0|a)}\prod\limits_{j_1=1}^N \frac{\Gamma_M(b_0+b_{j_1}|a)}{\Gamma_M(q+b_0+b_{j_1}|a)}
\prod\limits_{j_1<j_2}^N \frac{\Gamma_M(q+b_0+b_{j_1}+b_{j_2}|a)}{\Gamma_M(b_0+b_{j_1}+b_{j_2}|a)} \nonumber \\
&\times\prod\limits_{j_1<j_2<j_3}^N \frac{\Gamma_M(b_0+b_{j_1}+b_{j_2}+b_{j_3}|a)}{\Gamma_M(q+b_0+b_{j_1}+b_{j_2}+b_{j_3}|a)} \cdots,
\end{align}
until all the $N$ indices are exhausted. The function $\log\eta_{M,N}(q\,|a,\,b)$ has an important integral representation that follows from 
that of $\log\Gamma_M(w|a)$ in \eqref{keyRuij}.
\begin{theorem}[Existence and Structure]\label{main}
Given $M, N\in\mathbb{N}$ such that $M\leq N,$ the function
$\eta_{M,N}(q\,|a,\,b)$ is the Mellin transform of a probability
distribution on $(0, 1].$ Denote it by $\beta_{M, N}(a,b).$ Then,
\begin{equation}
{\bf E}\bigl[\beta_{M, N}(a,b)^q\bigr] = \eta_{M, N}(q\,|a,\,b),\;
\Re(q)>-b_0.
\end{equation}
The distribution $-\log\beta_{M, N}(a,b)$ is infinitely divisible on
$[0, \infty)$ and has the L\'evy-Khinchine decomposition for $\Re(q)>-b_0,$
\begin{equation}\label{LKH}
{\bf E}\Bigl[\exp\bigl(q\log\beta_{M, N}(a,b)\bigr)\Bigr] =
\exp\Bigl(\int\limits_0^\infty (e^{-tq}-1) e^{-b_0
t} \frac{
\prod\limits_{j=1}^N (1-e^{-b_j t})}{\prod\limits_{i=1}^M (1-e^{-a_i t})}
\frac{dt}{t} \Bigr).
\end{equation}
%%\begin{corollary}[Structure]\label{Structure}
$\log\beta_{M,N}(a,b)$ is absolutely continuous if and only if $M=N.$ If $M<N,$\\
$-\log\beta_{M,N}(a,b)$ is compound Poisson and
\begin{subequations}
\begin{align}
{\bf P}\bigl[\beta_{M,N}(a,b)=1\bigr] & =
\exp\Bigl(-\int\limits_0^\infty e^{-b_0 t}\frac{
\prod\limits_{j=1}^N (1-e^{-b_j t})}{\prod\limits_{i=1}^M (1-e^{-a_i t})}
\frac{dt}{t}\Bigr), \label{Pof1} \\
& = \exp\bigl(-(\mathcal{S}_N \log\Gamma_M)(0\,|a,\,b)\bigr). \label{Pof12}
\end{align}
\end{subequations}
\end{theorem}
It is worth emphasizing that the integral representation of $\log\eta_{M,N}(q\,|a,\,b)$ in \eqref{LKH} is the main result as it 
automatically implies that $\beta_{M, N}(a,b)$ is a valid probability distribution having
infinitely divisible logarithm, see Chapter 3 of \cite{SteVHar} for background material
on infinitely divisible distributions on $[0, \infty).$ 

The Mellin transform of Barnes beta distributions satisfies a function equation
that is inherited from that of the multiple gamma function and two remarkable
factorizations.
\begin{theorem}[Properties]\label{FunctEquat}
$1\leq M\leq N,$ $q\in\mathbb{C}-(-\infty, -b_0],$ $i=1\cdots M,$
\begin{equation}
\eta_{M, N}(q+a_i\,|\,a,\,b) =
%%\eta_{M, N}(q\,|\,a,\,b) \frac{\eta_{M, N}(a_i\,|\,a,\,b)}{\eta_{M-1, N}(q\,|\,\hat{a}_i,\,b)}, \\
\eta_{M, N}(q\,|\,a,\,b)\,\exp\bigl(-(\mathcal{S}_N
\log\Gamma_{M-1})(q\,|\,\hat{a}_i, b)\bigr). \label{fe1}
\end{equation}
Let $\Omega\triangleq \sum_{i=1}^M n_i \, a_i.$
\begin{align}
\eta_{M,N}(q\,|\,a, b)  = & \prod\limits_{k=0}^\infty \frac{\eta_{M-1,N}(q+k
a_i\,|\,\hat{a}_i, b)}{\eta_{M-1,N}(k
a_i\,|\,\hat{a}_i, b)}, \label{infinprod2} \\
\eta_{M,N}(q\,|\,a, b) = & \prod\limits_{n_1,\cdots ,n_M=0}^\infty \Bigl[
\frac{b_0+\Omega}{q+b_0+\Omega}\prod\limits_{j_1=1}^N \frac{q+b_0+b_{j_1}+\Omega}
{b_0+b_{j_1}+\Omega} \times \nonumber \\
& \times\prod\limits_{j_1<j_2}^N \frac{b_0+b_{j_1}+b_{j_2}+\Omega}
{q+b_0+b_{j_1}+b_{j_2}+\Omega} \times \nonumber \\
& \times
\prod\limits_{j_1<j_2<j_3}^N \frac{q+b_0+b_{j_1}+b_{j_2}+b_{j_3}+\Omega}{b_0+b_{j_1}+b_{j_2}+b_{j_3}+\Omega}\cdots\Bigr].
\label{infbarnesfac}
\end{align}
Probabilistically, these factorizations are equivalent to, respectively, 
\begin{align}
\beta_{M, N}(a,\,b) \overset{{\rm in \,law}}{=}&
\prod\limits_{k=0}^\infty \beta_{M-1, N}(\hat{a}_i,\,b_0+ka_i, \, \,b_1,\cdots, b_N), \\
\beta_{M, N}(a,\,b) \overset{{\rm in \,law}}{=}&
\prod\limits_{n_1,\cdots ,n_M=0}^\infty \beta_{0, N}(b_0+\Omega,\,\,b_1,\cdots, b_N). \label{probbarnesfac}
\end{align}
\end{theorem}
We note that the factorizations in \eqref{infinprod2} and \eqref{infbarnesfac} correspond to the Shintani and Barnes factorizations of the multiple gamma function, see \eqref{generalfactorization} and \eqref{barnes}, respectively.  
%%\section{}
The functional equation in \eqref{fe1} gives us the moments.
\begin{corollary}[Moments]
Assume $a_i=1.$ 
Let $k\in\mathbb{N}.$
\begin{align}
{\bf E}\bigl[\beta_{M, N}(a, b)^{k }\bigr] &  =
\exp\Bigl(-\sum\limits_{l=0}^{k-1} \bigl(\mathcal{S}_N
\log\Gamma_{M-1}\bigr)(l \,|\,\hat{a}_i, b)\Bigr), \label{posmom} \\
{\bf E}\bigl[\beta_{M, N}(a, b)^{-k }\bigr] & =
\exp\Bigl(\sum\limits_{l=0}^{k-1} \bigl(\mathcal{S}_N
\log\Gamma_{M-1}\bigr)(-(l+1) \,|\,\hat{a}_i, b)\Bigr), \; k<b_0. \label{negmom}
\end{align}
\end{corollary}
The scaling property in \eqref{scale} gives us the scaling invariance.
\begin{theorem}[Scaling invariance]\label{barnesbetascaling}
Let $\kappa>0.$ Then,
\begin{equation}
\beta^{\kappa}_{M, N}(\kappa\,a, \kappa\,b) \overset{{\rm in \,law}}{=}\beta_{M,
N}(a,\,b).
\end{equation}
\end{theorem}

The interested reader can find additional properties of Barnes beta distributions and many examples in \cite{Me13}.

The multiplication property in \eqref{multiply} gives us the structure of the Selberg integral probability distribution, as was first shown in \cite{Me14}.
Recall the classical Selberg integral. 
\begin{gather}
\int\limits_{[0,\,1]^l} \prod_{i=1}^l s_i^{\lambda_1}(1-s_i)^{\lambda_2}\, \prod\limits_{i<j}^l |s_i-s_j|^{-2/\tau} ds_1\cdots ds_l = \nonumber \\
\nonumber \\ \prod_{k=0}^{l-1}\frac{\Gamma(1-(k+1)/\tau)
\Gamma(1+\lambda_1-k/\tau)\Gamma(1+\lambda_2-k/\tau)}
{\Gamma(1-1/\tau)\Gamma(2+\lambda_1+\lambda_2-(l+k-1)/\tau)}, \label{Selberg}
\end{gather}
confer \cite{ForresterBook} for a modern treatment. We will assume for simplicity that $\lambda_i\geq 0$ and $\tau>1.$
In what follows we write $\tau$ as an abbreviation of $(1,\tau)$ in the list of parameters of the double gamma and $\beta_{2,2}.$
\begin{theorem}[Selberg Integral Probability Distribution]\label{BSM}
Define the function
\begin{align}
\mathfrak{M}(q\,|\,\tau,\lambda_1,\lambda_2)  \triangleq &
\Bigl(\frac{2\pi\,\tau^{\frac{1}{\tau}}}{\Gamma\bigl(1-1/\tau\bigr)}\Bigr)^q\;
%\tau^{\frac{q}{\tau}} (2\pi)^{q}\,\Gamma^{-q}\bigl(1-1/\tau\bigr)
\frac{\Gamma_2(1-q+\tau(1+\lambda_1)\,|\,\tau)}{\Gamma_2(1+\tau(1+\lambda_1)\,|\,\tau)}\times
\nonumber \\ & \times
\frac{\Gamma_2(1-q+\tau(1+\lambda_2)\,|\,\tau)}{\Gamma_2(1+\tau(1+\lambda_2)\,|\,\tau)}
\frac{\Gamma_2(-q+\tau\,|\,\tau)}{\Gamma_2(\tau\,|\,\tau)}
\times
\nonumber \\ & \times
\frac{\Gamma_2(2-q+\tau(2+\lambda_1+\lambda_2)\,|\,\tau)}{\Gamma_2(2-2q+\tau(2+\lambda_1+\lambda_2)\,|\,\tau)}
\label{M}
\end{align}
for $\Re(q)<\tau.$ Then, $\mathfrak{M}(q\,|\,\tau,\lambda_1,\lambda_2)$ is the Mellin transform of a
probability distribution $M_{(\tau, \lambda_1, \lambda_2)}$ on $(0,\infty),$
\begin{equation}\label{Mint}
\mathfrak{M}(q\,|\,\tau,\lambda_1,\lambda_2) = {\bf E}\bigl[M_{(\tau,\lambda_1,\lambda_2)}^q\bigr], \;\Re(q)<\tau,
\end{equation}
and its positive moments satisfy for $l<\tau$ %%($0\leq l <
%%\tau$) %%and negative ($l\in\mathbb{N}$) integral moments satisfy
\begin{equation}
{\bf E}\bigl[M^l_{(\tau, \lambda_1, \lambda_2)}\bigr]  =
\prod_{k=0}^{l-1} \frac{\Gamma(1-(k+1)/\tau)}{\Gamma(1-1/\tau)}
\frac{\Gamma(1+\lambda_1-k/\tau)\Gamma(1+\lambda_2-k/\tau)}
{\Gamma(2+\lambda_1+\lambda_2-(l+k-1)/\tau)}.\label{Selbergmomentsverified} 
\end{equation}
$\log M_{(\tau, \lambda_1, \lambda_2)}$ is absolutely continuous and
infinitely divisible. Define the distributions
\begin{align}
L \triangleq &\exp\bigl(\mathcal{N}(0,\,4\log 2/\tau)\bigr), \\ Y
\triangleq &\tau\,y^{-1-\tau}\exp\bigl(-y^{-\tau}\bigr)\,dy,\; y>0,\label{Ydist}
\end{align}
\emph{i.e.} $\log L$ is a zero-mean normal with variance $4\log
2/\tau$ and $Y$ is a power of the exponential. Let $X_1,\,X_2,\,X_3$ have the $\beta^{-1}_{2,
2}(\tau, b)$ distribution with the parameters %%($\beta^{-1}_{2,2}(\tau, b)\equiv\beta^{-1}_{2,2}((1, \tau), b))$
\begin{align}
X_1 &\triangleq \beta_{2,2}^{-1}\Bigl(\tau,
b_0=1+\tau+\tau\lambda_1,\,b_1=\tau(\lambda_2-\lambda_1)/2, \,
b_2=\tau(\lambda_2-\lambda_1)/2\Bigr),\\
X_2 & \triangleq \beta_{2,2}^{-1}\Bigl(\tau,
b_0=1+\tau+\tau(\lambda_1+\lambda_2)/2,\,b_1=1/2,\,b_2=\tau/2\Bigr),\\
X_3 & \triangleq \beta_{2,2}^{-1}\Bigl(\tau, b_0=1+\tau,\,
b_1=\frac{1+\tau+\tau\lambda_1+\tau\lambda_2}{2}, \,
b_2=\frac{1+\tau+\tau\lambda_1+\tau\lambda_2}{2}\Bigr).
\end{align}
Then,
\begin{equation}\label{Decomposition}
M_{(\tau, \lambda_1, \lambda_2)} \overset{{\rm in \,law}}{=} 2\pi\,
2^{-\bigl[3(1+\tau)+2\tau(\lambda_1+\lambda_2)\bigr]/\tau}\,\Gamma\bigl(1-1/\tau\bigr)^{-1}\,
L\,X_1\,X_2\,X_3\,Y.
\end{equation}
The Mellin transform is involution invariant under
\begin{equation}
\tau\rightarrow \frac{1}{\tau},\; q\rightarrow \frac{q}{\tau}, \; \lambda_i\rightarrow \tau\lambda_i.
\end{equation}
\begin{align}
\mathfrak{M}\bigl(\frac{q}{\tau}\,|\,\frac{1}{\tau},\tau\lambda_1,\tau\lambda_2\bigr) (2\pi)^{-\frac{q}{\tau}}\,\Gamma^{\frac{q}{\tau}}(1-\tau) \Gamma(1-\frac{q}{\tau}) = &
\mathfrak{M}(q\,|\,\tau,\lambda_1,\lambda_2) (2\pi)^{-q} \times \nonumber \\ & \times \Gamma^{q}(1-\frac{1}{\tau}) \Gamma(1-q).\label{involutionint}
\end{align}
\end{theorem} 
\begin{remark}\label{myremark}
The special case of $\lambda_1=\lambda_2=0$ of  Theorem \ref{BSM} first appeared in \cite{Me4}. The general case was first
considered in \cite{FLDR}, who gave an equivalent expression for the right-hand
side of \eqref{M} and verified \eqref{Selbergmomentsverified} without proving analytically that their formula corresponds to the Mellin transform of a probability distribution. The first proof of the existence of the Selberg integral distribution in full generality was given in \cite{MeIMRN},
where we also discovered the decomposition in \eqref{Decomposition}, followed by a new,  purely probabilistic proof 
of  \eqref{Decomposition} in \cite{Me14}.
The involution invariance of the Mellin transform in the equivalent form of self-duality, see \eqref{selfdual} below,  was first discovered in the special
case of $\lambda_1=\lambda_2=0$ in \cite{FLDR}. We extended it to the general case in the form of \eqref{involutionint} in \cite{Me14},
followed by the general form of self-duality in \cite{FLD}. The interested reader can find additional information about the Selberg integral
distribution such as negative moments, asymptotic expansion, moment determinacy questions, functional equations, and infinite factorizations
in \cite{MeIMRN}.
\end{remark}

\section{A New Family of Barnes Beta Distributions}
\noindent
Let $M\in\mathbb{N},$ $a=(a_1,\cdots, a_{M}),$ and $b=(b_0, b_1,\cdots,b_{M-1}),$ all assumed to be positive.
In other words, $N=M-1$ in the sense of Barnes beta distributions. 
Define
\begin{equation}
\eta_{M, M-1}(q|a, b) \triangleq \exp\Bigl( \bigl(\mathcal{S}_{M-1}
\log\Gamma_M\bigr)(q\,|a,\,b) - \bigl(\mathcal{S}_{M-1} \log\Gamma_M\bigr)(0\,|a,\,b) \Bigr). \label{etaL}
%%= & \eta_{M, M-1}(q|a, b).
\end{equation}
For example, in the case of $M=2$ we have
\begin{equation}
\eta_{2, 1}(q|a, b) =
\frac{\Gamma_2(q+b_0\,|\,a)}{\Gamma_2(b_0\,|\,a)}
\frac{\Gamma_2(b_0+b_1\,|\,a)}{\Gamma_2(q+b_0+b_1\,|\,a)}.
\end{equation}
Such products were first discovered in \cite{Kuz} and \cite{KuzPar} and then studied in depth in \cite{LetSim} in the context of
the Mellin transform of certain functionals of the stable L\'evy process. %%We will know show how to 
\begin{theorem}[Existence and Structure]\label{mainsine}
Assume
\begin{equation}
%%|a|-|b|>0.
\Re(q)>-b_0.
\end{equation}
$\eta_{M, M-1}(q\,|a,\,b)$ is the Mellin transform of a probability
distribution $\beta_{M, M-1}(a,b)$  on $(0, \infty).$
\begin{equation}
{\bf E}\bigl[\beta_{M, M-1}(a,b)^q\bigr] = \eta_{M, M-1}(q\,|a,\,b). 
%|a|-|b|>\Re(q)>-b_0.
\end{equation}
The distribution $\log\beta_{M, M-1}(a,b)$ is infinitely divisible and absolutely continuous on
$\mathbb{R}$ and has the L\'evy-Khinchine decomposition %%for $|a|-|b|>\Re(q)>-b_0,$
\begin{align}
{\bf E}\Bigl[\exp\bigl(q\log\beta_{M, M-1}(a,b)\bigr)\Bigr] = 
\exp&\Bigl(
\int\limits_0^\infty (e^{-tq}-1+qt) e^{-b_0
t} \frac{
\prod\limits_{j=1}^{M-1} (1-e^{-b_j t})}{\prod\limits_{i=1}^M (1-e^{-a_i t})}
\frac{dt}{t} +\nonumber \\&+ q
\int\limits_0^\infty \Bigl[\frac{e^{-t}}{t} \frac{\prod\limits_{j=1}^{M-1} b_j}{\prod\limits_{i=1}^{M} a_i}
-e^{-b_0
t}\frac{
\prod\limits_{j=1}^{M-1} (1-e^{-b_j t})}{\prod\limits_{i=1}^M (1-e^{-a_i t})}\Bigr]
dt
\Bigr).\label{LKHsine}
\end{align} 
$\eta_{M, M-1}(q\,|a,\,b)$ satisfies the functional equation in \eqref{fe1} and moment formulas in \eqref{posmom}
and \eqref{negmom}. Given $\kappa>0,$ the scaling invariance property of $\beta_{M, M-1}(a,b)$ is
\begin{equation}\label{scalinvgen}
\beta^{\kappa}_{M, M-1}(\kappa\,a, \kappa\,b) \overset{{\rm in \,law}}{=}\kappa^{\prod\limits_{j=1}^{M-1} b_j/\prod\limits_{i=1}^M a_i} \;\beta_{M, M-1}(a,\,b).
\end{equation}
\end{theorem}
\begin{theorem}[Asymptotics]\label{newetaasympt}
Given  $|\arg(q)|<\pi,$
\begin{equation}\label{ourasymN}
\eta_{M, M-1}(q\,|\,b) = \exp\Bigl(\bigl(\prod\limits_{j=1}^{M-1} b_j/\prod\limits_{i=1}^M a_i\bigr) q\log(q) +
O(q)\Bigr), \;q\rightarrow\infty.
\end{equation}
The Stieltjes moment problem for $\beta_{M, M-1}(a,b)$ is determinate (unique solution) iff 
\begin{equation}\label{condition}
\prod\limits_{j=1}^{M-1} b_j \leq 2\prod\limits_{i=1}^M a_i  .
\end{equation}
\end{theorem}

It is also interesting to look at the ratio of two independent Barnes beta distributions of type $(M, M-1)$ as this ratio
has remarkable factorization properties. Let 
\begin{equation}
\bar{b} \triangleq \bigl(\bar{b}_0, b_1, \cdots b_{M-1}\bigr)
\end{equation}
for some fixed $\bar{b}_0>0$
and define
\begin{equation}
\beta_{M, M-1}(a,b, \bar{b})\triangleq \beta_{M, M-1}(a,b)\,\beta^{-1}_{M, M-1}(a,\bar{b}).
\end{equation}
Then, the Mellin transform of %%multiple sine distributions 
$\beta_{M, M-1}(a,b, \bar{b})$ 
satisfies %%a function equation and 
two factorizations and scaling invariance that are similar to those of Barnes beta distributions.
\begin{theorem}[Properties]\label{FunctEquatSine}
Let $M\in\mathbb{N},$ $\bar{b}_0>\Re(q)>-b_0,$ $i=1\cdots M,$  $\Omega\triangleq \sum_{i=1}^M n_i \, a_i.$
\begin{align}
\eta_{M, M-1}(q\,|\,a, b, \bar{b})  = & \prod\limits_{k=0}^\infty \frac{\eta_{M-1,M-1}(q+k
a_i\,|\,\hat{a}_i, b)}{\eta_{M-1,M-1}(k
a_i\,|\,\hat{a}_i, b)} \,\frac{\eta_{M-1,M-1}(-q+k
a_i\,|\,\hat{a}_i, \bar{b})}{\eta_{M-1,M-1}(k
a_i\,|\,\hat{a}_i, \bar{b})}, \label{infinprod2L} \\
\eta_{M, M-1}(q\,|\,a, b, \bar{b}) = & \prod\limits_{n_1,\cdots ,n_M=0}^\infty \Bigl[
\frac{b_0+\Omega}{q+b_0+\Omega}\,\frac{\bar{b}_0+\Omega}{-q+\bar{b}_0+\Omega}
\times \nonumber \\
& \times
\prod\limits_{j_1=1}^{M-1} \frac{q+b_0+b_{j_1}+\Omega}
{b_0+b_{j_1}+\Omega} \,\frac{-q+\bar{b}_0+b_{j_1}+\Omega}
{\bar{b}_0+b_{j_1}+\Omega}
\times \nonumber \\
& \times\prod\limits_{j_1<j_2}^{M-1} \frac{b_0+b_{j_1}+b_{j_2}+\Omega}
{q+b_0+b_{j_1}+b_{j_2}+\Omega} \, \frac{\bar{b}_0+b_{j_1}+b_{j_2}+\Omega}
{-q+\bar{b}_0+b_{j_1}+b_{j_2}+\Omega} \cdots\Bigr].\label{infinprod1L}
%%\prod\limits_{j_1<j_2<j_3}^N \frac{q+b_0+b_{j_1}+b_{j_2}+b_{j_3}+\Omega}{b_0+b_{j_1}+b_{j_2}+b_{j_3}+\Omega}\cdots\Bigr].
\end{align}
Probabilistically, these factorizations are equivalent to, respectively, 
\begin{align}
\beta_{M, M-1}(a,\,b, \bar{b}) \overset{{\rm in \,law}}{=}&
\prod\limits_{k=0}^\infty \beta_{M-1, M-1}(\hat{a}_i,\,b_0+ka_i,\,\,b_1,\cdots, b_{M-1})\times \nonumber \\ & \times \beta_{M-1, M-1}^{-1}(\hat{a}_i,\,\bar{b}_0+ka_i,\,b_1,\cdots, b_{M-1}), \label{probshin}\\
\beta_{M, M-1}(a,\,b, \bar{b}) \overset{{\rm in \,law}}{=}&
\prod\limits_{n_1,\cdots ,n_M=0}^\infty \beta_{0, M-1}(b_0+\Omega,\,b_1,\cdots, b_{M-1})\times \nonumber \\ & \times\beta_{0, M-1}^{-1}(\bar{b}_0+\Omega, \,\,b_1,\cdots, b_{M-1}). \label{probbarnes}
\end{align}
Let $\kappa>0.$ Then,
\begin{equation}\label{scalinvL}
\beta^{\kappa}_{M, M-1}(\kappa\,a, \kappa\,b, \kappa\,\bar{b}) \overset{{\rm in \,law}}{=}\beta_{M, M-1}(a,\,b, \bar{b}).
\end{equation}
\end{theorem}
\begin{remark}
The multiple sine function \cite{KurKoya} is defined by
\begin{equation}
S_M(w|a) \triangleq \frac{\Gamma_M(|a|-w|a)^{(-1)^M}}{\Gamma_M(w|a)},
\end{equation}
where $M=0,1,2\cdots,$ $|a|=\sum_{i=1}^M a_i,$ and $a=(a_1,\cdots, a_M)$ are fixed positive
constants. %%This function plays an important role in ...
Let 
\begin{equation}\label{b0spec}
\bar{b}_0 = |a|-\sum_{j=0}^{M-1} b_j>0.
\end{equation}
Then, for this particular choice of $\bar{b}_0$ we have the interesting identity
\begin{align}
\eta_{M, M-1}(q|a, b, \bar{b}) & = \exp\Bigl(\bigl(\mathcal{S}_{M-1}
\log S_M\bigr)(0\,|a,\,b) - \bigl(\mathcal{S}_{M-1} \log S_M\bigr)(q\,|a,\,b)\Bigr), \nonumber \\
& =  \frac{S_M(b_0|a)}{S_M(q+b_0|a)}\prod\limits_{j_1=1}^{M-1} \frac{S_M(q+b_0+b_{j_1}|a)}{S_M(b_0+b_{j_1}|a)} \times\nonumber \\
&\times
\prod\limits_{j_1<j_2}^{M-1} \frac{S_M(b_0+b_{j_1}+b_{j_2}|a)}{S_M(q+b_0+b_{j_1}+b_{j_2}|a)} \times\nonumber \\
&\times\prod\limits_{j_1<j_2<j_3}^{M-1} \frac{S_M(q+b_0+b_{j_1}+b_{j_2}+b_{j_3}|a)}{S_M(b_0+b_{j_1}+b_{j_2}+b_{j_3}|a)} \cdots,
\end{align}
until all $M-1$ indices are exhausted. 
\end{remark}

We will illustrate the general theory with the special cases of $\beta_{1,0}$ and $\beta_{2,1}.$ Let $M=1.$
\begin{align}
{\bf E}[\beta_{1,0}(a, b)^q] = & \frac{\Gamma_1(q+b_0\,|\,a)}{\Gamma_1(b_0\,|\,a)}, \\
= & a^{\frac{q}{a}} \frac{\Gamma(\frac{q+b_0}{a})}{\Gamma(\frac{b_0}{a})}, %%\frac{\Gamma_1(-q+\bar{b}_0)}{\Gamma_1(\bar{b}_0)}
\end{align}
by \eqref{gamma1} so that $\beta_{1,0}$ is a Fr\'echet distribution. This distribution plays an important role
in the structure of both Selberg and Morris integral probability distributions. For example, the $Y$ distribution in \eqref{Ydist} is
\begin{equation}
Y = \tau^{1/\tau}\,\beta^{-1}_{1,0}(a=\tau, b_0=\tau). 
\end{equation}
In particular, if $\bar{b}_0=a-b_0$ as in \eqref{b0spec},
\begin{align}
{\bf E}[\beta_{1,0}(a, b, \bar{b})^q] = & \frac{\Gamma(\frac{q+b_0}{a})}{\Gamma(\frac{b_0}{a})} \frac{\Gamma(1-\frac{q+b_0}{a})}
{\Gamma(1-\frac{b_0}{a})}, \\
= & \frac{\sin(\frac{\pi b_0}{a})}{\sin(\frac{\pi (q+b_0)}{a})}.
\end{align}
Now, let $M=2,$ $a=(a_1, a_2),$  $b=(b_0, b_1),$ $\bar{b}=(a_1+a_2-b_0-b_1, b_1).$
\begin{align}
{\bf E}[\beta_{2,1}(a, b)^q] = &\frac{\Gamma_2(q+b_0\,|\,a)}{\Gamma_2(b_0\,|\,a)}\frac{\Gamma_2(b_0+b_1\,|\,a)}{\Gamma_2(q+b_0+b_1\,|\,a)}, \\
{\bf E}[\beta_{2,1}(a, b, \bar{b})^q] = & \frac{S_2(b_0\,|\,a)}{S_2(q+b_0\,|\,a)}\frac{S_2(q+b_0+b_1\,|\,a)}{S_2(b_0+b_1\,|\,a)}.
\end{align}
For reader's convenience, the proofs of Theorems \ref{mainsine} -- \ref{FunctEquatSine} are deferred to Section 6.

\section{Application to the Morris Integral}
\noindent In this section we will consider the problem of finding a positive probability distribution having the property
that its positive integral moments are given by the Morris integral. Throughout this section we let $\tau>1$
and restrict $\lambda_1, \lambda_2\geq 0$ for simplicity.  As we did in Section 2, we write $\tau$ as an abbreviation of 
$a=(1,\tau)$ in the list of parameters of the double gamma function and $\beta_{2,2}(a, b).$ Recall the Morris integral, 
see Chapter 4 of \cite{ForresterBook}.
\begin{gather}
\int\limits_{[-\pi,\,\pi]^n} \prod\limits_{l=1}^n  e^{ i \theta_l(\lambda_1-\lambda_2)/2}  |1+e^{ i\theta_l}|^{\lambda_1+\lambda_2} \,
\prod\limits_{k<l}^n |e^{ i \theta_k}-e^{ i\theta_l}|^{-2/\tau} \,d\theta = \nonumber \\ =(2\pi)^n\prod\limits_{j=0}^{n-1} \frac{\Gamma(1+\lambda_1+\lambda_2- \frac{j}{\tau})\,\Gamma(1-\frac{(j+1)}{\tau})}{\Gamma(1+\lambda_1- \frac{j}{\tau})\,\Gamma(1+\lambda_2- \frac{j}{\tau})\,\Gamma(1-\frac{1}{\tau})} \label{morris2}.
\end{gather}
We wish to construct a probability distribution such that its $n$th moment coincides with the right-hand side of \eqref{morris2} for $n<\tau.$
The interest in such a construction will become clear in the next section.
\begin{theorem}[Existence and Properties]\label{theoremcircle}
The function 
\begin{align}\label{thefunction}
\mathfrak{M}(q\,|\tau,\,\lambda_1,\,\lambda_2)=&\frac{(2\pi\tau^{\frac{1}{\tau}})^q}{\Gamma^q\bigl(1-\frac{1}{\tau}\bigr)}
\frac{\Gamma_2(\tau(\lambda_1+\lambda_2+1)+1-q\,|\,\tau)}{\Gamma_2(\tau(\lambda_1+\lambda_2+1)+1\,|\,\tau)}
\frac{\Gamma_2(-q+\tau\,|\,\tau)}{\Gamma_2(\tau\,|\,\tau)}\times \nonumber \\ & \times
\frac{\Gamma_2(\tau(1+\lambda_1)+1\,|\,\tau)}{\Gamma_2(\tau(1+\lambda_1)+1-q\,|\,\tau)}
\frac{\Gamma_2(\tau(1+\lambda_2)+1\,|\,\tau)}{\Gamma_2(\tau(1+\lambda_2)+1-q\,|\,\tau)}
\end{align}
reproduces the product in Eq. \eqref{morris2} when $q=n<\tau$
and is the Mellin transform of the distribution
\begin{align}
M_{(\tau, \lambda_1, \lambda_2)} = & \frac{2\pi \tau^{1/\tau}}{\Gamma(1-1/\tau)} \,\beta^{-1}_{22}(\tau, b_0=\tau,\,b_1=1+\tau \lambda_1, \,b_2=1+\tau \lambda_2) \times \nonumber \\ & \times
\beta_{1,0}^{-1}(\tau, b_0=\tau(\lambda_1+\lambda_2+1)+1), \label{thedecompcircle}
%%Y(\tau, a, b),
\end{align}
where $\beta^{-1}_{22}(a, b)$ is the inverse Barnes beta of type $(2,2)$ and $\beta^{-1}_{1,0}(a, b)$ is the independent inverse Barnes beta of type $(1,0).$ In particular, $\log M_{(\tau, \lambda_1, \lambda_2)}$ is infinitely divisible, the Stieltjes moment problem for $M^{-1}_{(\tau, \lambda_1, \lambda_2)}$ is determinate (unique solution), and the negative moments
are given by
\begin{equation}
{\bf E}[M^{-n}_{(\tau, \lambda_1, \lambda_2)}] = (2\pi)^{-n} \prod\limits_{j=0}^{n-1} \frac{\Gamma(1+\lambda_1+\frac{(j+1)}{\tau}) \,\Gamma(1+\lambda_2+\frac{(j+1)}{\tau})\Gamma(1-\frac{1}{\tau})}{\Gamma(1+\lambda_1+\lambda_2+\frac{(j+1)}{\tau})\,\Gamma(1+\frac{j}{\tau})}. 
\end{equation}
The Mellin transform is involution invariant under
\begin{equation}\label{invtranscircle}
\tau\rightarrow \frac{1}{\tau},\; q\rightarrow \frac{q}{\tau}, \; \lambda_i\rightarrow \tau\lambda_i.
\end{equation}
\begin{align}
\mathfrak{M}\bigl(\frac{q}{\tau}\,|\,\frac{1}{\tau},\tau\lambda_1,\tau\lambda_2\bigr) (2\pi)^{-\frac{q}{\tau}}\,\Gamma^{\frac{q}{\tau}}(1-\tau) \Gamma(1-\frac{q}{\tau}) = &
\mathfrak{M}(q\,|\,\tau,\lambda_1,\lambda_2) (2\pi)^{-q} \times \nonumber \\ & \times \Gamma^{q}(1-\frac{1}{\tau}) \Gamma(1-q). \label{invcircle}
\end{align}
In the special case of $\lambda_1=\lambda_2=0,$ we have
\begin{gather}
\mathfrak{M}(q\,|\tau, 0, 0)=\frac{(2\pi)^q}{\Gamma^q\bigl(1-\frac{1}{\tau}\bigr)}\Gamma\bigl(1-\frac{q}{\tau}\bigr), \\
M_{(\tau, 0, 0)} = \frac{2\pi \tau^{1/\tau}}{\Gamma(1-1/\tau)}\beta_{1,0}^{-1}(\tau, b_0=\tau).
\end{gather}
\end{theorem}
\begin{remark}
The special case of $\lambda_1=\lambda_2=0$ was first treated in \cite{FyoBou} and corresponds to
the Dyson integral. The involution invariance in this case, in its equivalent self-duality form, see \eqref{selfdual} below, was first discovered
in \cite{FLDR}.
\end{remark}
The proofs are deferred to Section 6.

\section{Maximum Distribution of the Gaussian Free Field on the Unit Interval and Circle}
\noindent
In this section we will formulate precise conjectures about the distribution of the maximum of the discrete 2D gaussian
free field (GFF) on the unit interval and circle with a non-random logarithmic potential.  We will not attempt to review the GFF here but rather refer the reader to
\cite{FyoBou} for the circle case and to \cite{FLDR} and Section 3 of \cite{Menon} for the interval case. Suffice it to say that our approach to the 
GFF construction is essentially based on the construction of \cite{MRW}, \cite{BM1}, \cite{BM}.

\begin{definition}[GFF on the interval]
Let $V_\varepsilon(x)$ be a centered gaussian process on $[0, \,1]$ with the covariance
\begin{align}\label{covr}
{\bf{Cov}}\left[V_{\varepsilon}(u), \,V_{\varepsilon}
(v)\right] &  =
\begin{cases}
- \, 2\log |u-v|, \, \varepsilon\leq|u-v|\leq 1, \\
2\left(1-\log \varepsilon -\frac{|u-v|}{\varepsilon}\right),\, |u-v|\leq \varepsilon.
\end{cases}
\end{align}
The limit $\lim\limits_{\varepsilon\rightarrow 0} V_{\varepsilon}(x)$ is
called the continuous GFF on the interval $x\in[0, \,1]$ and its discretized version $V_\varepsilon(x_i),$  
$x_i=i\varepsilon,$ $i=0\cdots N,$ $\varepsilon=1/N$ is the discrete GFF on the interval. 
\end{definition}
We note that in applications, see \cite{Menon} for example, the GFF construction arises is a slightly more general form of
\begin{align}
%%{\bf{E}} \left[V_{\varepsilon}(u)\right] & =  -\frac{\mu}{2} \,
%%\left(\kappa-\log\varepsilon\right), \label{meank} \\
{\bf{Cov}}\left[V_{\varepsilon}(u), \,V_{\varepsilon}
(v)\right]  = & 
\begin{cases}\label{covk}
 -
2 \, \log|u-v|, \, \varepsilon\ll|u-v|\leq 1,  \\
 2
\left(\kappa-\log\varepsilon\right),\, u=v,
\end{cases} \nonumber \\
& + O(\varepsilon),
\end{align}
where $\kappa\geq0$ is some fixed constant and the details of regularization are relegated to the $O(\varepsilon)$ term.
It is worth emphasizing that the choice of covariance regularization for $|u-v|\leq \varepsilon$ has no effect on
distribution of the maximum, so long as the variance behaves as in \eqref{covk},  due to Theorem 6 in \cite{BM1} as explained below.
%%  the law of the total mass  so that the  is not affected by this choice. 
The same remark applies to the GFF on the circle so we give its definition in the more general form.
\begin{definition}[GFF on the circle]
Let $V_\varepsilon(\psi)$ be a centered gaussian process with the covariance
\begin{align}
%%{\bf{E}} \left[\omega_{\mu, \varepsilon}(\psi)\right] & =  -\frac{\mu}{2} \,
%%\left(\kappa-\log\varepsilon\right), \label{meank} \\
{\bf{Cov}}\left[V_{\varepsilon}(\psi), \,V_{\varepsilon}
(\xi)\right]  = &
\begin{cases}\label{covkc}
 -
2 \, \log|e^{i\psi}-e^{i\xi}|, \,  |\xi-\psi|\gg \varepsilon,  \\
 2
\left(\kappa-\log\varepsilon\right), \psi=\xi,
\end{cases} \nonumber \\
& + O(\varepsilon),
\end{align}
where $\kappa\geq 0$ is some fixed constant. The limit $\lim\limits_{\varepsilon\rightarrow 0} V_{\varepsilon}(\psi)$ is
called the GFF on the circle $\psi\in[-\pi, \pi)$ and its discretized version $V_\varepsilon(\psi_j),$  
$\psi_j=j\varepsilon,$ $j=-N/2\cdots N/2,$ $\varepsilon=1/N$ is the discrete GFF on the circle. 
\end{definition}
The existence of such objects follows from the general theory of \cite{RajRos} as shown in \cite{BM1} and \cite{BM}. 
\begin{definition}[Problem formulation for the interval]
Let $\lambda_1,\,\lambda_2\geq 0.$  
Let 
\begin{equation}
N=1/\varepsilon
\end{equation}
so we can think of $V_\varepsilon(x)$ as being defined at $x_i=i\varepsilon,$ $i=0\cdots N.$
What is the distribution of 
\begin{equation}
V_N\triangleq\max\{V_{\varepsilon}(x_i)+\lambda_1\,\log(x_i)+\lambda_2\,\log(1-x_i),\,i=1\cdots N\}
\end{equation}
in the form of an asymptotic expansion in $N$ in the limit $N\rightarrow \infty?$
\end{definition}
%%The gaussian free field is a fascinating mathematical object that 
\begin{definition}[Problem formulation for the circle]
Let $\alpha\geq 0.$ Let
\begin{equation}
N=2\pi/\varepsilon
\end{equation}
and think of $V_\varepsilon(\psi)$ as being defined at $\psi_j=j\varepsilon,$ $j=-N/2\cdots N/2.$
What is the distribution of 
\begin{equation}
V_N\triangleq\max\{V_{\varepsilon}(\psi_j)+2\alpha\log|1+e^{i\psi_j}|,\,j=-N/2\cdots N/2\}
\end{equation}
in the form of an asymptotic expansion in $N$ in the limit $N\rightarrow \infty?$
\end{definition}

We will consider the case of the GFF on the interval first. Define the critical Selberg integral probability 
distribution in terms of  the Selberg integral distribution in Theorem \ref{BSM} by the formula
\begin{definition}[Critical Selberg Integral Distribution]
\begin{equation}\label{Selbergcrit}
M_{(\tau=1,\lambda_1,\lambda_2)}\triangleq \lim\limits_{\tau\downarrow 1} \frac{\Gamma\bigl(1-1/\tau\bigr)}{2\pi}\,M_{(\tau,\lambda_1,\lambda_2)}.
\end{equation}
\end{definition}
This limit exists by \eqref{Decomposition}. We will refer to $M_{(\tau=1,\lambda_1,\lambda_2)}$ as the critical Selberg integral distribution.
\begin{conjecture}[Maximum of the GFF on the Interval]\label{maxint}
The leading asymptotic term in the expansion of the Laplace transform of $V_N$ in $N$ is 
\begin{equation}
{\bf E}[e^{q\,V_N}] \approx e^{q(2\log N-(3/2)\log\log N+{\rm const})}\,{\bf E}\bigl[M^q_{(\tau=1,\lambda_1,\lambda_2)}
\bigr],\;N\rightarrow\infty.
\end{equation}
Probabilistically, let 
\begin{align}
X_1 &\triangleq \beta_{2,2}^{-1}\Bigl(\tau=1,
b_0=2+\lambda_1,\,b_1=(\lambda_2-\lambda_1)/2, \,
b_2=(\lambda_2-\lambda_1)/2\Bigr),\\
X_2 & \triangleq \beta_{2,2}^{-1}\Bigl(\tau=1,
b_0=2+(\lambda_1+\lambda_2)/2,\,b_1=1/2,\,b_2=1/2\Bigr),\\
X_3 & \triangleq \beta_{2,2}^{-1}\Bigl(\tau=1, b_0=2,\,
b_1=(2+\lambda_1+\lambda_2)/2, \,
b_2=(2+\lambda_1+\lambda_2)/2\Bigr),\\
Y & \triangleq \beta_{1,0}^{-1}\bigl(\tau=1, b_0=1\bigr)
\end{align}
be as in Theorem \ref{BSM} with $\tau=1.$  Then, as $N\rightarrow \infty,$
\begin{align}
V_N = & 2\log N-\frac{3}{2}\log\log N+{\rm const}+\mathcal{N}(0,\,4\log 2) + \log X_1+\log X_2+\log X_3+\nonumber \\ & +
\log Y+\log Y'+o(1),
\end{align}
where $Y'$ is an independent copy of $Y.$ %%$O(1)$ is non-random, and corrections are of order $o(1).$ 
\end{conjecture}
%%This means that $V_N$ equals in law the sum of six independent random variables.
\begin{remark}
This conjecture at the level of the Mellin transform is due to \cite{FLDR} in the case of $\lambda_1=\lambda_2=0$ and \cite{FLD} for
general $\lambda_1$ and $\lambda_2.$ Our probabilistic re-formulation of their conjecture is new. We also note that 
the Mellin transform of the Barnes beta distribution $\beta_{2,2}$ with $a=(1,1)$ can be expressed in terms of the Barnes $G$ function, see
\cite{Genesis}, as follows.
\begin{align}
{\bf E}\bigl[\beta_{2,2}^q(1, b_0, b_1, b_2)\bigr] = & \frac{G(b_0)}{G(q+b_0)}
\frac{G(q+b_0+b_1)}{G(b_0+b_1)}
\frac{G(q+b_0+b_2)}{G(b_0+b_2)}\times \nonumber \\ & \times
\frac{G(b_0+b_1+b_2)}{G(q+b_0+b_1+b_2)}.
\end{align}
In the special case of $\lambda_1=\lambda_2=0,$ $X_1$ degenerates,  $X_3$ becomes Pareto, and 
\begin{equation}
X_2 = \beta_{2,2}^{-1}\bigl(1,
b_0=2,\,b_1=1/2,\,b_2=1/2\bigr),
\end{equation}
see Section 6 of \cite{Me14} for details.
\end{remark}

Similarly, to formulate our conjecture for the maximum of the GFF on the circle, we need to define the critical Morris integral
probability distribution. 
\begin{definition}[Critical Morris Integral Distribution]
\begin{equation}
M_{(\tau=1,\lambda_1,\lambda_2)}\triangleq \lim\limits_{\tau\downarrow 1} \frac{\Gamma\bigl(1-1/\tau\bigr)}{2\pi}\,M_{(\tau,\lambda_1,\lambda_2)},
\end{equation}
where $M_{(\tau,\lambda_1,\lambda_2)}$ is as in Theorem \ref{theoremcircle}.
\end{definition}
This limit exists by \eqref{thedecompcircle}. We will refer to $M_{(\tau=1,\lambda_1,\lambda_2)}$ as the critical Morris integral distribution.
\begin{conjecture}[Maximum of the GFF on the Circle]\label{maxintcircle}
The leading asymptotic term in the expansion of the Laplace transform of $V_N$ in $N$ is 
\begin{equation}
{\bf E}[e^{q\,V_N}] \approx e^{q(2\log N-(3/2)\log\log N+{\rm const})}\,{\bf E}\bigl[M^q_{(\tau=1,\alpha,\alpha)}
\bigr],\;N\rightarrow\infty.
\end{equation}
Probabilistically, let 
\begin{align}
X & \triangleq \beta_{2,2}^{-1}\bigl(\tau=1,
b_0=1,\,b_1=1+\alpha,\,b_2=1+\alpha\bigr),\\
Y & \triangleq \beta_{1,0}^{-1}(\tau=1, b_0=2\alpha+2), \\
Y' & \triangleq \beta_{1,0}^{-1}\bigl(\tau=1, b_0=1\bigr)
\end{align}
be as in Theorem \ref{theoremcircle} with $\tau=1.$  Then,
\begin{equation}
V_N = 2\log N-\frac{3}{2}\log\log N+{\rm const}+\log X+ \log Y+\log Y'+o(1).
\end{equation}
%%where $Y'$ is an independent copy of $Y.$ 
If $\alpha=0,$ 
\begin{equation}
V_N = 2\log N-\frac{3}{2}\log\log N+{\rm const}+ \log Y+\log Y'+o(1),
\end{equation}
where 
\begin{equation}
Y\overset{{\rm in \,law}}{=}Y'=\beta_{1,0}^{-1}\bigl(\tau=1, b_0=1\bigr).
\end{equation}
%%In both cases $O(1)$ is non-random and corrections are of order $o(1).$ 
\end{conjecture}
\begin{remark}
This conjecture is due to \cite{FyoBou} in the case of $\alpha=0.$ The extension to general $\alpha$ is new.
\end{remark}
In the rest of this section we will give a heuristic derivation of our conjectures. 
Our approach is based on the freezing hypothesis and calculations of \cite{FyoBou} and \cite{FLDR} as well as our conjectures
relating the distribution of the exponential functional of the GFF on the interval and circle to 
the Selberg and Morris integral distributions, respectively. To this end, we need to briefly recall
the principal result of \cite{BM1} concerning the existence of the gaussian multiplicative chaos measure (also known
as the limit lognormal measure, see \cite{secondface}, \cite{Lan}), which we are stating here in slightly more general form.
\begin{theorem}[Multiplicative Chaos on the Interval]\label{multichaos}
%%Mandelbrot `72, Kahane `85, 
Let $0\leq \beta<1.$ The exponential functional of the GFF on the interval converges weakly a.s. as $\varepsilon\rightarrow 0$ to
a non-degenerate limit random measure
\begin{gather}
e^{-\beta^2(\kappa-\log\varepsilon)}\int_a^b e^{\beta V_{\varepsilon}(u)} \, du\longrightarrow M_{\beta}(a, b), \\
{\bf{E}}[M_{\beta}(a, b)]=|b-a|.
\end{gather}
 %%(non-degenerate), $M_{\mu}(t,t+\tau) \overset{{\rm in \,law}}{=} M_{\mu}(0,\tau)$ (stationary)
The moments of the (generalized) total mass of the limit measure are given by the Selberg integral: let $n<1/\beta^2,$ %%$I\subset [0, \,1],$
\begin{equation}
{\bf{E}} \Bigl[\Bigl(\int_0^1 s^{\lambda_1}(1-s)^{\lambda_2} \,
M_\beta(ds)\Bigr)^n\Bigr] = \int\limits_{[0,\,1]^n} \prod_{i=1}^n
s_i^{\lambda_1}(1-s_i)^{\lambda_2} \prod_{i<j}^n
|s_i-s_j|^{-2\beta^2} ds_1\cdots ds_n.
\end{equation}
\end{theorem}
We constructed a probability distribution having the required moments, see Theorem \ref{BSM}, which leads us to
the following conjecture, originally due to \cite{Me4} for $\lambda_1=\lambda_2=0$ and due to \cite{FLDR} (at the level of the
Mellin transform, see Remark \ref{myremark} at the end of Section 2) and \cite{MeIMRN} in general.
\begin{conjecture}[Law of Total Mass]\label{ourmainconj}
Let $M_{(\tau,\lambda_1,\lambda_2)}$ be the Selberg integral distribution in \eqref{Mint}. 
%%I conjecture
\begin{equation}
M_{(\tau,\lambda_1,\lambda_2)} \overset{{\rm in \,law}}{=}  \int_0^1 s^{\lambda_1} (1-s)^{\lambda_2} 
\,M_\beta(ds),\; \tau=1/\beta^2>1.
\end{equation}
\end{conjecture}
%%Given these preliminaries, we can state our result for the GFF on the interval. 
Now, 
following \cite{FLDR}, define the exponential functional
\begin{equation}\label{Zdef}
Z_{\lambda_1,\lambda_2,\varepsilon}(\beta) \triangleq \sum\limits_{i=1}^N  x_i^{\beta\lambda_1}(1-x_i)^{\beta\lambda_2} e^{\beta V_\varepsilon(x_i)}.
\end{equation}
Using the identity
\begin{equation}\label{key}
{\bf P} \bigl(V_N < s \bigr) = \lim\limits_{\beta\rightarrow \infty} {\bf E}\Bigl[
\exp\Bigl(-e^{-\beta s}\,Z_{\lambda_1,\lambda_2,\varepsilon}(\beta)/C\Bigr)\Bigr],
\end{equation}
which is applicable to any sequence of random variables and an arbitrary \emph{$\beta$-independent} constant $C,$
the distribution of the maximum is reduced to the Laplace transform of the exponential functional in the limit $\beta\rightarrow \infty.$
Now, by Theorem \ref{multichaos}, it is known that for $0<\beta<1$ the exponential functional converges\footnote{
Theorem 6 in \cite{BM1} shows that the laws of the total mass of the continuous and discrete multiplicative chaos measures are the same
provided  the $\varepsilon$ parameter in \eqref{covr} coincides with the discretization step.}  as 
$N\rightarrow \infty$ to the (generalized) total mass of the multiplicative chaos (limit lognormal) measure on the unit interval,
which is conjectured to be given by the Selberg integral distribution, resulting %%, see Conjecture \ref{ourmainconj}.
in the approximation\footnote{The validity of approximating the finite $N$ quantity with the $N\rightarrow\infty$
limit is discussed in \cite{FyoBou}.}
\begin{equation}\label{keyapprox}
Z_{\lambda_1,\lambda_2,\varepsilon}(\beta) \approx N^{1+\beta^2}\, e^{\beta^2\kappa}\,M_{(\tau,\beta\lambda_1,\beta\lambda_2)},\;N\rightarrow\infty,
\end{equation}
where
\begin{equation}
\tau = 1/\beta^2,\; 0<\beta<1.
\end{equation}
Next, we recall the involution invariance of the Mellin transform of the Selberg integral distribution, see \eqref{involutionint}.
%%As was first discovered in \cite{FLDR} and later verified in \cite{Me14} using this decomposition and the scaling invariance of Barnes beta distributions, 
%%see Theorem 4.2 in \cite{Me14}, the Mellin transform of $M_{\mu,\lambda_1,\lambda_2}$ satisfies a self-duality equation, \emph{i.e.}
%%is involution invariant under
%\begin{equation}
%\tau\rightarrow \frac{1}{\tau},\; q\rightarrow \frac{q}{\tau}, \; \lambda_i\rightarrow \tau\lambda_i.
%\end{equation}
Denoting the Mellin transform as in Theorem \ref{BSM},
\begin{equation}
\mathfrak{M}(q\,|\,\tau,\lambda_1,\lambda_2) = {\bf E}\bigl[M_{(\tau,\lambda_1,\lambda_2)}^q\bigr], \;\Re(q)<\tau,
\end{equation}
%%Then, see Corollary 5.7 in \cite{Me14},
we have
\begin{align}
\mathfrak{M}\bigl(\frac{q}{\tau}\,|\,\frac{1}{\tau},\tau\lambda_1,\tau\lambda_2\bigr) (2\pi)^{-\frac{q}{\tau}}\,\Gamma^{\frac{q}{\tau}}(1-\tau) \Gamma(1-\frac{q}{\tau}) = &
\mathfrak{M}(q\,|\,\tau,\lambda_1,\lambda_2) (2\pi)^{-q}\times \nonumber \\ & \times\Gamma^{q}(1-\frac{1}{\tau}) \Gamma(1-q). \label{involutionsec}
\end{align}
Hence, introducing the function
\begin{equation}\label{Func}
F(q\,|\,\beta, \lambda_1,\lambda_2) \triangleq \mathfrak{M}\bigl(\frac{q}{\beta}\,|\,\frac{1}{\beta^2},\beta\lambda_1,\beta\lambda_2\bigr) (2\pi)^{-\frac{q}{\beta}}\,\Gamma^{\frac{q}{\beta}}(1-\beta^2) \Gamma(1-\frac{q}{\beta}),
\end{equation}
one observes that it satisfies the identity
\begin{equation}\label{selfdual}
F(q\,|\,\beta, \lambda_1,\lambda_2) = F(q\,\big|\,\frac{1}{\beta}, \lambda_1,\lambda_2),
\end{equation}
first discovered in \cite{FLDR} in the special case of $\lambda_1=\lambda_2=0,$ then formulated in general in the form of \eqref{involutionsec} in
\cite{Me14}, and in \cite{FLD} in this form. 
On the other hand, as shown in \cite{FLDR}, one has the general identity
\begin{equation}\label{generalidentity}
\int_\mathbb{R} e^{yq} \frac{d}{dy}\exp\Bigl(-e^{-\beta y}X\Bigr)  dy = X^{\frac{q}{\beta}} \,\Gamma(1-\frac{q}{\beta}),\;\Re(q)<0, \,X>0.
\end{equation}
Letting 
\begin{equation}\label{X}
X = Z_{\lambda_1,\lambda_2,\varepsilon}(\beta) \frac{e^{\kappa}\,\Gamma(1-\beta^2)}{2\pi}
\end{equation}
one sees by means of \eqref{keyapprox} that for $0<\beta<1,$
\begin{equation}
{\bf E}\bigl[X^{\frac{q}{\beta}}\bigr]\approx e^{q\kappa(\beta+\frac{1}{\beta})} N^{q(\beta+\frac{1}{\beta})}\,\mathfrak{M}\bigl(\frac{q}{\beta}\,|\,\frac{1}{\beta^2},\beta\lambda_1,\beta\lambda_2\bigr) (2\pi)^{-\frac{q}{\beta}}\,\Gamma^{\frac{q}{\beta}}(1-\beta^2),\; N\rightarrow\infty,
\end{equation}
so that
\begin{equation}\label{rhs}
{\bf E}\bigl[X^{\frac{q}{\beta}}\bigr]\, \Gamma(1-\frac{q}{\beta}) \approx e^{q\kappa(\beta+\frac{1}{\beta})} N^{q(\beta+\frac{1}{\beta})}\,F(q\,|\,\beta, \lambda_1,\lambda_2),\;
N\rightarrow\infty.
\end{equation}
On the other hand, by letting the constant $C$ in \eqref{key} to be taken to be
\begin{equation}
C = \frac{2\pi}{e^{\kappa}\Gamma(1-\beta^2)},
\end{equation}
and combining \eqref{key} with \eqref{generalidentity} , one obtains
\begin{equation}
{\bf E}[e^{q V_N}] = \lim_{\beta\rightarrow \infty} \Bigl[{\bf E}\bigl[X^{\frac{q}{\beta}}\bigr] \Gamma(1-\frac{q}{\beta})\Bigr].
\end{equation}
The right-hand side of this equation has only been determined for $0<\beta<1,$ see \eqref{rhs}.
Due to the self-duality of the right-hand side, one assumes that it gets frozen at $\beta=1,$ as first formulated in \cite{FLDR}.
\begin{conjecture}[The Freezing Hypothesis]
Let $\beta>1.$
\begin{equation}
{\bf E}\bigl[X^{\frac{q}{\beta}}\bigr] \Gamma(1-\frac{q}{\beta}) = 
{\bf E}\bigl[X^{\frac{q}{\beta}}\bigr] \Gamma(1-\frac{q}{\beta})\Big|_{\beta=1}.
\end{equation}
\end{conjecture}
%%The formula for $X$ in \eqref{X} means that the constant $C$ in \eqref{key} is taken to be
One must note however that $C$ is not \emph{$\beta$-independent}. It is argued in \cite{YO} and \cite{FLDR2}  
that the $\Gamma(1-\beta^2)$ term 
shifts the maximum by  $-(3/2) \log \log N,$ see also \cite{Ding}. %%I think that the $\beta^2 \kappa$ term only
%%affects the shift of the maximum distribution as it is $N$-independent and is not singular at $\beta=1.$ 
Overall, we then obtain by \eqref{rhs},
%%\begin{equation}
%%\approx  \lim_{\beta\rightarrow \infty} \Bigl[e^{q\kappa(\beta+\frac{1}{\beta})} N^{q(\beta+\frac{1}{\beta})}\,F(q\,|\,\beta, \lambda_1,\lambda_2)\Bigr],\;N\rightarrow\infty.
%%\end{equation}
\begin{equation}
{\bf E}[e^{q\,V_N}] \approx e^{q(2\log N-(3/2)\log\log N+\,{\rm const})}\,F(q\,|\,\beta=1, \lambda_1,\lambda_2),\;
N\rightarrow\infty
\end{equation}
for some constant.\footnote{As remarked in \cite{FyodSimm}, this procedure only determines the distribution of the maximum up
to a constant term.}
Finally, recalling definitions of the critical Selberg integral distribution in \eqref{Selbergcrit} and of $F(q\,|\,\beta, \lambda_1,\lambda_2)$
in \eqref{Func}, and appropriately adjusting the constant, 
\begin{equation}
{\bf E}[e^{q\,V_N}] \approx e^{q(2\log N-(3/2)\log\log N+\,{\rm const})}\,{\bf E}\Bigl[M^q_{(\tau=1,\lambda_1,\lambda_2)}\Bigr]\Gamma(1-q),
\end{equation}
so that $Y'$ comes from the $\Gamma(1-q)$ factor and has the same law as $Y.$

The argument for the GFF on the circle goes through verbatim so we will only point out the key steps and omit redundant details.
Define the exponential functional 
\begin{equation}\label{Zdefc}
Z_{\alpha,\varepsilon}(\beta) = \sum\limits_{j=-N/2}^{N/2}  |1+e^{i\psi_j}|^{2\alpha\,\beta} e^{\beta V_\varepsilon(\psi_j)}.
\end{equation}
To describe its limit as $N\rightarrow \infty$ we need to compute the distribution of the generalized total mass of the limit
lognormal measure on the circle. The following result follows from the general theory of multiplicative chaos measures of \cite{K2}.
\begin{theorem}[Multiplicative Chaos on the Circle]\label{multichaoscircle}
%%Mandelbrot `72, Kahane `85, 
Let $0\leq \beta<1.$ The exponential functional of the GFF on the circle converges weakly a.s. as $\varepsilon\rightarrow 0$ to
a non-degenerate limit random measure
\begin{gather}
e^{-\beta^2(\kappa-\log\varepsilon)}\int_\phi^\psi e^{\beta V_{\varepsilon}(\theta)} \, d\theta\longrightarrow M_{\beta}(\phi, \psi), \\
{\bf{E}}[M_{\beta}(\phi, \psi)]=|\psi-\phi|.
\end{gather}
 %%(non-degenerate), $M_{\mu}(t,t+\tau) \overset{{\rm in \,law}}{=} M_{\mu}(0,\tau)$ (stationary)
The moments of the (generalized) total mass of the limit measure are given by the Morris integral: let $n<1/\beta^2,$ %%$I\subset [0, \,1],$
\begin{align}
{\bf{E}} \Bigl[\Bigl( \int_{-\pi}^{\pi} e^{i\psi \frac{\lambda_1-\lambda_2}{2}} \, |1+e^{i\psi}|^{\lambda_1+\lambda_2}\, dM_{\beta}(\psi)\Bigr)^n\Bigr] = & \int\limits_{[-\pi,\,\pi]^n} \prod\limits_{l=1}^n  e^{ i \theta_l\frac{\lambda_1-\lambda_2}{2}}  |1+e^{ i\theta_l}|^{\lambda_1+\lambda_2} \nonumber\\
& \times
\prod\limits_{k<l}^n |e^{ i \theta_k}-e^{ i\theta_l}|^{-2\beta^2} \,d\theta.
\end{align}
\end{theorem}
We constructed a probability distribution having the required moments, see Theorem \ref{theoremcircle}, which leads us to
the following conjecture.
\begin{conjecture}[Law of Total Mass]\label{ourmainconjcircle}
Let $M_{(\tau,\lambda_1,\lambda_2)}$ be as in \eqref{thedecompcircle}. Let 
\begin{equation}
\lambda_1=\lambda_2=\alpha,\footnote{This restriction is necessary as $M_{(\tau,\lambda_1,\lambda_2)}$ is real-valued whereas  
$\int_{-\pi}^{\pi} e^{i\psi \frac{\lambda_1-\lambda_2}{2}} \, |1+e^{i\psi}|^{\lambda_1+\lambda_2}\, dM_{\beta}(\psi)$ is not in general, unless $\lambda_1=\lambda_2.$  The problem of determining the law of $\int_{-\pi}^{\pi} e^{i\psi \frac{\lambda_1-\lambda_2}{2}} \, |1+e^{i\psi}|^{\lambda_1+\lambda_2}\, dM_{\beta}(\psi)$ for $\lambda_1\neq\lambda_2$ is left to future research.}
\end{equation}
%%I conjecture
\begin{equation}
M_{(\tau,\alpha,\alpha)} \overset{{\rm in \,law}}{=}   \int_{-\pi}^{\pi} |1+e^{i\psi}|^{2\alpha}\, dM_{\beta}(\psi),\; \tau=1/\beta^2>1.
\end{equation}
\end{conjecture}
This conjecture for $\lambda_1=\lambda_2=0$ is due to \cite{FyoBou}, the general case is original to this paper.

%%so that
It follows that the exponential functional in \eqref{Zdefc} can be approximated by
\begin{equation}\label{keyapproxc}
Z_{\alpha,\varepsilon}(\beta) \approx \bigl(\frac{N}{2\pi}\bigr)^{1+\beta^2}\, e^{\beta^2\kappa}\,M_{(\tau,\beta\alpha,\beta\alpha)},\;N\rightarrow\infty.
\end{equation}
The rest of the argument is the same as for the GFF on the interval, with \eqref{invcircle} replacing \eqref{involutionint}.

\section{Proofs}
\noindent We begin with a key lemma on the action of $\mathcal{S}_{M-1}$ on polynomials.
\begin{lemma}[Main lemma]\label{MyLemma}
Let $f(t)$ be a fixed function and define the associated Bernoulli polynomials by
\begin{equation}\label{Bdefa}
B^{(f)}_{k}(x) \triangleq \frac{d^m}{dt^k}|_{t=0} \bigl[f(t)
e^{-xt}\bigr].
\end{equation}
\begin{align}
\bigl(\mathcal{S}_{M-1} B^{(f)}_{k}\bigr)(q\,|\,b) & = 0,\; k<M-1, \label{id1}\\
\bigl(\mathcal{S}_{M-1}B^{(f)}_{M-1}\bigr)(q\,|\,b) & = \bigl(\mathcal{S}_{M-1} B^{(f)}_{M-1}\bigr)(0\,|\,b), \label{id2}\\
\bigl(\mathcal{S}_{M-1} B^{(f)}_{M}\bigr)(q\,|\,b) & - \bigl(\mathcal{S}_{M-1} B^{(f)}_{M}\bigr)(0\,|\,b) 
 = -qf(0)\,M! \prod_{j=1}^{M-1} b_j. \label{id3}
\end{align}     
\end{lemma}       
\begin{proof}
The identities in \eqref{id1} and \eqref{id2} were first established in \cite{Me13}. As the argument for 
\eqref{id3} is similar, we will reproduce the original argument here for completeness. 
Define the function
$g(t),$
\begin{equation}\label{gfunction}
g(t) \triangleq f(t) e^{-qt} e^{-b_0
t}\prod\limits_{j=1}^N (1-e^{-b_j t}).
\end{equation}
Then, using the identity
\begin{align}
e^{-(q+b_0) t}\prod\limits_{j=1}^N (1-e^{-b_j t}) = & e^{-(q+b_0) t}\sum\limits_{p=0}^N (-1)^p
\sum\limits_{k_1<\cdots<k_p=1}^N
\exp\bigl(-(b_{k_1}+\cdots+b_{k_p})t\bigr), \nonumber \\  = & \bigl(\mathcal{S}_{N} e^{-xt}\bigr)(q|b), \label{theid}
\end{align}
one obtains
\begin{align}
g^{(k)}(0) & = %%\sum\limits_{m=0}^n \binom{n}{m} B^{(f)}_{n-m}(q)
%%\frac{d^{m+r}}{dt^{m+r}}|_{t=0}\bigl[e^{-b_0 t}\prod\limits_{j=1}^N
%%(1-e^{-b_j t})\bigr], \label{line1} \\
(\mathcal{S}_{N}
B^{(f)}_k)(q\,|\,b), \label{line2orig} \\
&= 0,\;{\rm if}\;k<N, \label{deriv1}\\
&= f(0)\,N! \prod\limits_{j=1}^N b_j,\;{\rm
if}\;k=N.\label{deriv2}
\end{align}
%%Letting $f(t)=1,$ $N=M-1,$ we obtain  \eqref{id1} and\eqref{id2}, and letting $f(t)=f_M(t|a)$ as in \eqref{fdef},
%%we get \eqref{id4} and \eqref{id5}. 
To verify \eqref{id3}, letting $N=M-1,$ $k=M,$ one observes that \eqref{line2orig} implies the identity
\begin{equation}
(\mathcal{S}_{M-1}
B^{(f)}_M)(q\,|\,b) = f(0)\frac{d^M}{dt^M}|_{t=0}\prod\limits_{j=1}^{M-1} (1-e^{-b_j t}) +
M! \prod\limits_{j=1}^{M-1} b_j \,B^{(f)}_{1}(q+b_0).
\end{equation}
Hence, %%the alternating sums that occurs in \eqref{id3} and \eqref{id6} can be written as
\begin{gather}
\bigl(\mathcal{S}_{M-1} B^{(f)}_{M}(x)\bigr)(q\,|\,b) - \bigl(\mathcal{S}_{M-1} B^{(f)}_{M}(x)\bigr)(0\,|\,b) 
%%\bigl(\mathcal{S}_{M-1} B^{(f)}_{M}(x)\bigr)(-q\,|\,\bar{b}) -\bigl(\mathcal{S}_{M-1} B^{(f)}_{M}(x)\bigr)(0\,|\,\bar{b}),\nonumber \\
= \nonumber \\ M! \prod\limits_{j=1}^{M-1} b_j \Bigl( B^{(f)}_{1}(q+b_0)- B^{(f)}_{1}(b_0)%%+ B^{(f)}_{1}(-q+\bar{b}_0)- B^{(f)}_{1}(\bar{b}_0)
\Bigr).
\label{alter}
\end{gather}
Finally, it remains to notice that $B^{(f)}_{1}(x)$ satisfies the identity
\begin{equation}
B^{(f)}_{1}(x) - B^{(f)}_{1}(x+y) = f(0)\,y,
\end{equation}
and \eqref{id3} follows. 
\qed
\end{proof}
\begin{proof}[Proof of Theorem \ref{mainsine}]
%%Let $\eta_{M}(q|a, b)$ be defined by \eqref{etaL}. 
We wish to establish the identity
\begin{equation}\label{keyid}
\eta_{M, M-1}(q|a, b) = \exp\Bigl(
\int\limits_0^\infty \frac{dt}{t}\Bigl[(e^{-tq}-1)e^{-b_0 t}\frac{\prod\limits_{j=1}^{M-1} (1-e^{-b_j t})}{\prod\limits_{i=1}^M (1-e^{-a_i t})} + 
q e^{-t} \frac{\prod\limits_{j=1}^{M-1} b_j}{\prod\limits_{i=1}^M a_i}\Bigr]
\Bigr).
\end{equation}
Once it is established, \eqref{LKHsine} follows immediately by adding and subtracting $qt$ in the integrand,
which allows us to split the integral in \eqref{keyid} into two individual integrals and thereby bring it 
to the required L\'evy-Khinchine form. To verify \eqref{keyid} we need to recall the Ruijsenaars formula in Theorem \ref{R}.
We see from Lemma \ref{MyLemma} that for any $k=0\cdots M-1$ we have the identity
\begin{align}
\bigl(\mathcal{S}_{M-1} B_{M, k}(x|a)\bigr)(q\,|\,b) & - \bigl(\mathcal{S}_{M-1} B_{M, k}(x|a)\bigr)(0\,|\,b) 
%%\bigl(\mathcal{S}_{M-1} B_{M, k}(x|a)\bigr)(-q\,|\,\bar{b}) - \nonumber \\&-\bigl(\mathcal{S}_{M-1} B_{M, k}(x|a)\bigr)(0\,|\,\bar{b})
 = 0.
\end{align}
It follows upon substituting \eqref{keyRuij} into \eqref{etaL} and using \eqref{id3} that the only non-vanishing terms are
\begin{align}
\eta_{M, M-1}(q|a, b) =& \exp\Bigl( \int\limits_0^\infty \frac{dt}{t} \Bigl[\frac{
\bigl(\mathcal{S}_{M-1} e^{-xt}\bigr)(q\,|\,b)-\bigl(\mathcal{S}_{M-1} e^{-xt}\bigr)(0\,|\,b)}{\prod\limits_{i=1}^M (1-e^{-a_i t})}
%%\bigl(\mathcal{S}_{M-1} e^{-xt}\bigr)(-q\,|\,\bar{b}),\nonumber \\ &-\bigl(\mathcal{S}_{M-1} e^{-xt}\bigr)(0\,|\,\bar{b})
+\nonumber \\ & + qe^{-t}  \frac{\prod\limits_{j=1}^{M-1} b_j}{\prod\limits_{i=1}^M a_i}
\Bigr]
\Bigl),
\end{align}
and the result follows from \eqref{theid}. Now, we have the obvious identities
\begin{gather}
\int\limits_0^\infty e^{-b_0 t} \frac{\prod\limits_{j=1}^{M-1} (1-e^{-b_j t})}{\prod\limits_{i=1}^M (1-e^{-a_i t})} \frac{dt}{t} = \infty, \\
\int\limits_0^\infty e^{-b_0 t} \frac{\prod\limits_{j=1}^{M-1} (1-e^{-b_j t})}{\prod\limits_{i=1}^M (1-e^{-a_i t})} dt = \infty,
\end{gather}
which imply that $\log\beta_{M, M-1}(a,b)$ is absolutely continuous and supported on $\mathbb{R}$ by
Theorem 4.23 and Proposition 8.2 in Chapter 4 of \cite{SteVHar}, respectively. 
The scaling invariance in \eqref{scalinvgen}
is a corollary of \eqref{scale} and \eqref{id3}. \qed
\end{proof}
\begin{proof}[Proof of Theorem \ref{newetaasympt}]
The asymptotic expansion of $\log\Gamma_M(w\,|\, a)$ in \eqref{asym} consists of
the leading term $-B_{M, M}(w\,|\,a)\,\log(w)/M!$ plus a polynomial remainder. Lemma \ref{MyLemma}
shows that the remainder term contributes $O(q)$ to $\log\eta_{M, M-1}(q|a, b).$ It remains to
show that as $q\rightarrow \infty,$
\begin{equation}
\bigl(\mathcal{S}_{M-1} B_{M, M}(x\,|\,a)\,\log(x)\bigr)(q|b) %%-  \bigl(\mathcal{S}_{M-1} B_{M, M}(x)\,\log(x)\bigr)(0|b) = 
= - M!\bigl(\prod\limits_{j=1}^{M-1} b_j/\prod\limits_{i=1}^M a_i\bigr) q\log(q) +
O(q).
\end{equation}
%%Expanding the $\log$ in powers of $1/q,$ 
Slightly generalizing the calculation in Lemma \ref{MyLemma}, let
\begin{equation}\label{gfunctionasym}
g(t) \triangleq f(t) e^{-qt} \frac{d^r}{dt^r}\bigl[e^{-b_0
t}\prod\limits_{j=1}^{M-1} (1-e^{-b_j t})\bigr].
\end{equation}
Then,
\begin{align}
g^{(n)}(0) & = \sum\limits_{m=0}^n \binom{n}{m} B^{(f)}_{n-m}(q)
\frac{d^{m+r}}{dt^{m+r}}|_{t=0}\bigl[e^{-b_0 t}\prod\limits_{j=1}^{M-1}
(1-e^{-b_j t})\bigr], \label{line1} \\
& = (-1)^r \sum\limits_{p=0}^{M-1} (-1)^p
\sum\limits_{k_1<\cdots<k_p=1}^{M-1} \bigl(b_0+\sum b_{k_j}\bigr)^r \,
B^{(f)}_n\bigl(q+b_0+\sum b_{k_j}\bigr). \label{line2} 
\end{align}
In our case $n=M,$ $f(t)$ as in \eqref{fdef}, and the expression in \eqref{line2} is the coefficient of
$1/q^r$ that one gets by expanding $\bigl(\mathcal{S}_{M-1} B_{M, M}(x\,|\,a)\,\log(x)\bigr)(q|b)$ in powers of $1/q.$ 
By \eqref{line1} we can restrict ourselves to $m+r\geq M-1.$  On the other hand, 
the overall power of $q$ is $M-m-r$ so that we are only interested in $M-m-r\geq 1,$
as the other terms contribute $O(1).$  Hence, $m=M-r-1,$ and the contribution of such terms
is $O(q),$
\begin{equation}
\bigl(\mathcal{S}_{M-1} B_{M, M}(x\,|\,a)\,\log(x)\bigr)(q|b) = \log(q)\bigl(\mathcal{S}_{M-1} B_{M, M}(x\,|\,a)\,\bigr)(q|b)+  O(q),
\end{equation}
and the result follows from \eqref{id3}. Finally, the asymptotic behavior in \eqref{ourasymN} coincides with 
the asymptotic behavior of generalized gamma distributions and \eqref{condition} follows from
the known solution to the Stieltjes moment problem for these distributions, see \cite{Stoyanov}.\qed
\end{proof}
\begin{proof}[Proof of Theorem \ref{FunctEquatSine}]
%%This is an immediate corollary of the fundamental functional equation \eqref{feq}. 
The factorizations in \eqref{infinprod2L} and
\eqref{infinprod1L} are corollaries of Lemma \ref{MyLemma} and Shintani and Barnes factorizations of the multiple gamma functions, see \eqref{generalfactorization} and \eqref{barnes}, respectively. A direct proof can be given as follows. The Mellin transform of $\beta_{M, M-1}(a, b, \bar{b})$ is
\begin{equation}\label{keyidb}
\eta_{M, M-1}(q|a, b, \bar{b}) = \exp\Bigl(
\int\limits_0^\infty \frac{dt}{t}\Bigl((e^{-tq}-1)e^{-b_0 t}+(e^{tq}-1)e^{-\bar{b}_0 t}\Bigr)\frac{\prod\limits_{j=1}^{M-1} (1-e^{-b_j t})}{\prod\limits_{i=1}^M (1-e^{-a_i t})}\Bigr).
\end{equation}
Let $i=M$ without any loss
of generality. The formula in \eqref{keyidb} can be written in the form
\begin{align}
\eta_{M, M-1}(q|a, b, \bar{b})  = &e^{
\int\limits_0^\infty \frac{dt}{t}\sum\limits_{k=0}^\infty \Bigl[(e^{-tq}-1)e^{-(b_0+ka_M) t} + (e^{tq}-1) e^{-(\bar{b}_0+ka_M)t} \Bigr] \prod\limits_{j=1}^{M-1} \frac{(1-e^{-b_j t})}{(1-e^{-a_j t})}}, \nonumber \\
= & \prod\limits_{k=0}^\infty  \exp\Bigl(
\int\limits_0^\infty \frac{dt}{t} (e^{-tq}-1)e^{-(b_0+ka_M) t} \prod\limits_{j=1}^{M-1} \frac{(1-e^{-b_j t})}{(1-e^{-a_j t})}\Bigr)
\times \nonumber \\  &\times \exp\Bigl(
\int\limits_0^\infty \frac{dt}{t} (e^{tq}-1) e^{-(\bar{b}_0+ka_M)t} \prod\limits_{j=1}^{M-1} \frac{(1-e^{-b_j t})}{(1-e^{-a_j t})}\Bigr)
, \label{keyid2}
\end{align}
which is exactly the same as the expression in \eqref{infinprod2L} if one recalls \eqref{LKH} and the following identity
that we first noted in \cite{Me13},
\begin{equation}
\eta_{M, N}(q\,|\,a,\,b_0+x, b_1\cdots b_N)\,\eta_{M, N}(x\,|\,a,\,b) = \eta_{M, N}(q+x\,|\,a,\,b), \; x>0,
\end{equation}
where $b=(b_0,b_1,\cdots b_N).$ The expression in \eqref{keyid2} is equivalent to \eqref{probshin}.
If we now apply \eqref{probbarnesfac} to each Barnes beta factor in \eqref{probshin}, we obtain \eqref{probbarnes},
which is equivalent to \eqref{infinprod1L}. The scaling invariance in \eqref{scalinvL} follows from \eqref{scalinvgen}.
\qed
\end{proof}
\begin{proof}[Proof of Theorem \ref{theoremcircle}]
Recall the functional equation of the double gamma function, see \eqref{feq} with $M=2,$ and the definition of $\Gamma_1(w|a)$ in
\eqref{gamma1}. Using the functional equation repeatedly, we obtain for $a=(1,\tau)$ and $k\in\mathbb{N},$
\begin{equation}\label{repeated}
\frac{\Gamma_2(z+1-k\,|\,\tau)}{\Gamma_2(z+1\,|\,\tau)} = \bigl(\frac{1}{2\pi\tau}\bigr)^{k/2} \tau^{\sum\limits_{j=0}^{k-1} (z-j)/\tau}\;
\prod\limits_{j=0}^{k-1} \Gamma\bigl(\frac{z}{\tau}-\frac{j}{\tau}\bigr).
\end{equation}
We now apply this equation to each of the four ratios of double gamma functions in \eqref{thefunction}  with $q=n,$ 
which results in \eqref{morris2}. Now, the inverse Barnes beta distribution $\beta^{-1}_{2,2}(a, b)$ with parameters
$a=(1,\tau)$ and $b=(\tau, 1+\tau\lambda_1, 1+\tau\lambda_2)$ has the Mellin transform
\begin{align}
{\bf E}\bigl[\beta^{-q}_{2,2}(\tau, \tau, 1+\tau\lambda_1, 1+\tau\lambda_2)\bigr] = &\frac{\Gamma_2(-q+\tau\,|\,\tau)}{\Gamma_2(\tau\,|\,\tau)}
\frac{\Gamma_2(\tau(1+\lambda_1)+1\,|\,\tau)}{\Gamma_2(\tau(1+\lambda_1)+1-q\,|\,\tau)}\times \nonumber \\ & \times
\frac{\Gamma_2(\tau(1+\lambda_2)+1\,|\,\tau)}{\Gamma_2(\tau(1+\lambda_2)+1-q\,|\,\tau)}
\times \nonumber \\ & \times
\frac{\Gamma_2(\tau(\lambda_1+\lambda_2+1)+2-q\,|\,\tau)}{\Gamma_2(\tau(\lambda_1+\lambda_2+1)+2\,|\,\tau)}.
\end{align}
The difference between this expression and that in \eqref{thefunction} is in the last factor. Applying the functional
equation once again, we get
\begin{align}
\frac{\Gamma_2(\tau(\lambda_1+\lambda_2+1)+1-q\,|\,\tau)}{\Gamma_2(\tau(\lambda_1+\lambda_2+1)+1\,|\,\tau)}
=&\tau^{-\frac{q}{\tau}} 
\frac{\Gamma(\lambda_1+\lambda_2+1+\frac{1-q}{\tau})}{\Gamma(\lambda_1+\lambda_2+1+\frac{1}{\tau})}\times \nonumber \\ & \times
\frac{\Gamma_2(\tau(\lambda_1+\lambda_2+1)+2-q\,|\,\tau)}{\Gamma_2(\tau(\lambda_1+\lambda_2+1)+2\,|\,\tau)}.
\end{align}
Recalling the definition of the Mellin transform of the inverse Barnes beta  $\beta^{-1}_{1,0}(a, b)$ with parameters
$a=\tau$ and $b=1+\tau(1+\lambda_1+\lambda_2),$ 
\begin{equation}
{\bf E} \bigl[\beta^{-q}_{1,0}\bigl((\tau, 1+\tau(1+\lambda_1+\lambda_2)\bigr)\bigr] = 
 \tau^{-\frac{q}{\tau}} \frac{\Gamma(-\frac{q}{\tau}+1+\lambda_1+\lambda_2+\frac{1}{\tau})}{\Gamma(1+\lambda_1+\lambda_2+\frac{1}{\tau})},
\end{equation}
we see that the Mellin transform of the distribution $M_{(\tau, \lambda_1, \lambda_2)}$ in \eqref{thedecompcircle} coincides with the expression in \eqref{thefunction}.

The infinite divisibility of $\log M_{(\tau, \lambda_1, \lambda_2)}$ follows from that of $\log\beta^{-1}_{2,2}(a, b)$ and $\log\beta^{-1}_{1,0}(a, b).$
The determinacy of the Stieltjes moment problem for $M^{-1}_{(\tau, \lambda_1, \lambda_2)}$ follows from Theorem \ref{newetaasympt} as
$\beta_{2,2}(a, b)$ is compactly supported. In our case the condition for determinacy in \eqref{condition},
\begin{equation}
1\leq 2\tau,
\end{equation}
is satisfied as $\tau>1.$ The negative moments and the special case of $\lambda_1=\lambda_2=0$ follow from \eqref{repeated}. 
To prove the involution invariance in \eqref{invcircle}, we need to recall the scaling property of the multiple gamma function, see \eqref{scale}.
The transformation in \eqref{invtranscircle} corresponds to $\kappa=1/\tau.$ We note first that
\begin{equation}
\frac{1}{\tau}(1,\,\tau)=(1, \,\frac{1}{\tau})
\end{equation}
in the sense of the parameters of the double gamma function. We apply \eqref{scale} to each of the double gamma factors in \eqref{thefunction}
under the transformation in \eqref{invtranscircle}. For example, 
\begin{align}
\Gamma_2\bigl(\frac{1}{\tau}(\tau\lambda_1+\tau\lambda_2+1)+1-\frac{q}{\tau}\,\Big|\,\frac{1}{\tau}\bigr) = & \Gamma_2\Bigl(\frac{1}{\tau}\bigl(\tau(\lambda_1+\lambda_2+1)+1-q\bigr)\,\Big|\,\frac{1}{\tau}\Bigr), \nonumber \\
= & \bigl(\frac{1}{\tau}\bigr)^{-B_{2,2}(\tau(\lambda_1+\lambda_2+1)+1-q)/2} \times\nonumber \\ & \times \Gamma_2\Bigl(\tau(\lambda_1+\lambda_2+1)+1-q\,\Big|\,\tau\Bigr).
\end{align}
Using the identity
\begin{equation}\label{B22}
B_{2,2}(x\,|\,a) =
\frac{x^2}{a_1a_2}-\frac{x(a_1+a_2)}{a_1a_2}+\frac{a_1^2+3a_1a_2+a_2^2}{6a_1a_2},
\end{equation}
with $(a_1=1,\,a_2=\tau),$ we collect all terms and simplify to obtain
\begin{align}
\mathfrak{M}\bigl(\frac{q}{\tau}\,|\,\frac{1}{\tau},\tau\lambda_1,\tau\lambda_2\bigr) (2\pi)^{-\frac{q}{\tau}}\,\Gamma^{\frac{q}{\tau}}(1-\tau) = &
\mathfrak{M}(q\,|\,\tau,\lambda_1,\lambda_2)\Bigl(\frac{2\pi\tau^{\frac{1}{\tau}}}{\Gamma(1-\frac{1}{\tau})}\Bigr)^{-q} \times \nonumber \\ & \times
\frac{\Gamma_2(1-q\,|\tau)\Gamma_2(\tau\,|\tau)}{\Gamma_2(1\,|\tau)\Gamma_2(\tau-q\,|\tau)}.
\end{align}
It remains to observe that the functional equation of the double gamma function implies the identity
\begin{equation}
\tau^{-\frac{q}{\tau}} \frac{\Gamma_2(1-q\,|\tau)\Gamma_2(\tau\,|\tau)}{\Gamma_2(1\,|\tau)\Gamma_2(\tau-q\,|\tau)} = 
\frac{\Gamma(1-q)}{\Gamma(1-\frac{q}{\tau})},
\end{equation}
which gives the result. \qed
\end{proof}

\section{Conclusion and Open Questions}
\noindent
We introduced a new family $\beta_{M, M-1}(a, b),$ $M\in\mathbb{N},$ of Barnes beta distributions, \emph{i.e.} 
distributions whose Mellin transform is defined to be an intertwining product of ratios of multiple 
gamma functions of Barnes, that are supported on $(0, \infty).$ We showed infinite divisibility and 
absolute continuity of $\log\beta_{M, M-1}(a, b)$ by computing its Levy-Khinchine decomposition explicitly. 
We calculated the asymptotic expansion of the Mellin transform of $\beta_{M, M-1}(a, b)$ and %%used it to 
solved the Stieltjes moment problem for $\beta_{M, M-1}(a, b).$ We also showed that the 
ratio $\beta_{M, M-1}(a, b)\beta^{-1}_{M, M-1}(a, \bar{b})$ with different values of $b_0$ possesses remarkable 
infinite factorizations and its Mellin transform reduces in a special case to intertwining products 
of ratios of multiple sine functions. For application, 
we noted that $\beta_{2, 1}(a, b)$ coincides with the distribution, which was first discovered in \cite{Kuz} and \cite{KuzPar}
and recently studied in \cite{LetSim} in the context of laws of certain functionals of stable processes, 
and the distribution $\beta_{1, 0}(a, b)$ appears as a building block of both Selberg and Morris 
probability distributions. 

We reviewed the general theory of Barnes beta distributions to make our paper self-contained and 
then focused on the structure of the Selberg and Morris integral probability distributions. The construction
of the latter is original to this paper. In both cases we emphasized the connection with Barnes beta 
distributions of types $(1,0)$ and $(2,2)$ and showed as a corollary %%of the general theory 
that 
the Stieltjes moment problem for the negative moments of the Morris integral distribution, unlike that
of the Selberg integral distribution, is determinate. We noted that the self-duality 
of the Mellin transform of the Selberg integral distribution, which was established originally in \cite{FLD} and \cite{FLDR},
follows from its involution invariance that we established in \cite{Me14}. We showed that the self-duality 
of the Mellin transform of the Morris integral distribution, which was established originally in \cite{FLDR} in the special case
of the Dyson integral distribution, extends to the general Morris integral distribution and also follows from its
involution invariance, which we proved in this paper. For application, we examined the conjectures of 
\cite{FyoBou}, \cite{FLD}, and \cite{FLDR} about the maximum distribution of the gaussian free field on the interval and circle. 
We reviewed their calculations from the viewpoint of the gaussian multiplicative chaos theory and 
conjecturally expressed the law of the derivative martingale in both cases in terms of Barnes 
beta distributions of types $(1,0)$ and $(2,2).$ In particular, in the interval case, our probabilistic
reformulation of their conjecture is new and, in the circle case, our conjecture itself
is new and extends the original conjecture of \cite{FyoBou} by adding a non-random logarithmic potential to the gaussian free field.

There are at least two additional areas of applications of Barnes beta distributions that 
are beyond the scope of this paper and left to future research. First, the Selberg integral 
distribution appears conjecturally as a mod-gaussian limit of the exponential functional of a large class 
of known statistics that converge to a gaussian process with $\mathcal{H}^{1/2}(\mathbb{R})$ limiting covariance, 
see \cite{Menon}. Hence, the considerations of this paper can be naturally carried over to the 
maximum distribution of all such statistics. %%The same type of result can also be formulated 
%%for statistics that converge to gaussian processes having $\mathcal{H}^{1/2}(\mathbb{T})$ covariance, in which case 
%%the maximum distribution is conjecturally governed by the critical Morris integral distribution. 
Second, we showed in \cite{Me14} that a limit of Barnes beta distributions of type $\beta_{2M, 3M}$ 
as $M\rightarrow \infty$ can be used to approximate the completed Riemann zeta function.
It remains to see if this approximation is useful in verifying the
conjectures of \cite{YK} about extremal values of the Riemann zeta function on the critical line, which are related
to the Selberg and Morris integral distributions, \emph{i.e.} %%  and, in particular,
if the aforementioned limit can be related to $\beta_{2,2}$ distributions. 

\section*{Acknowledgments}
The author wants to express gratitude to Yan Fyodorov and Nickolas Simm for
helpful correspondence relating to ref. \cite{FyodSimm}. The author is also
thankful to the referees for many helpful suggestions.
%%\end{acknowledgements}

\end{document}